\documentclass[12pt]{amsart}

\usepackage{amsxtra,amssymb,amsthm,amsmath,amscd,url,eucal,xypic,listings,epsfig}\usepackage{supertabular}
\usepackage{relsize}
\usepackage[utf8]{inputenc}
\usepackage{fullpage}
\usepackage{dsfont,color}

\renewcommand{\leq}{\leqslant}
\renewcommand{\geq}{\geqslant}

\numberwithin{equation}{section}
\def\stacksum#1#2{{\stackrel{{\scriptstyle #1}}
{{\scriptstyle #2}}}}

\newcommand{\Cc}{\mathbf{C}}

\newcommand{\Gg}{\mathbf{G}}

\newcommand{\Zz}{\mathbf{Z}}

\newcommand{\Rr}{\mathbf{R}}

\newcommand{\Hc}{\mathcal{H}}
\newcommand{\Qq}{\mathbf{Q}}
\newcommand{\Fp}{\mathbf{F}}

\newcommand{\proba}{\text{\boldmath$P$}}
\newcommand{\expect}{\mathbf{E}}
\newcommand{\variance}{\mathbf{V}}
\newcommand{\charfun}{\mathds{1}}
\newcommand{\sy}{\mathfrak{S}}

\newcommand{\ra}{\rightarrow}
\newcommand{\lra}{\longrightarrow}

\newcommand{\fleche}[1]{\stackrel{#1}{\lra}}

\DeclareMathOperator{\frob}{\mathrm{Fr}}
\DeclareMathOperator{\Gal}{Gal}

\DeclareMathOperator{\Tr}{Tr}

\DeclareMathOperator{\disc}{disc}

\newcommand{\eps}{\varepsilon}

\newcommand{\cheb}[1]{c({{#1}})}
\newcommand{\chebp}[1]{c'({{#1}})}

\newcommand{\scheb}[1]{c_2({{#1}})}

\DeclareMathSymbol{\gena}{\mathord}{letters}{"3C}
\DeclareMathSymbol{\genb}{\mathord}{letters}{"3E}

\def\sumb{\mathop{\sum \Bigl.^{\flat}}\limits}

\def\dblsum{\mathop{\sum \sum}\limits}

\theoremstyle{plain}
\newtheorem{theorem}{Theorem}[section]

\newtheorem{lemma}[theorem]{Lemma}
\newtheorem{corollary}[theorem]{Corollary}

\newtheorem{proposition}[theorem]{Proposition}

\theoremstyle{remark}
\newtheorem{remark}[theorem]{Remark}

\theoremstyle{definition}

\newtheorem{definition}[theorem]{Definition}
\newtheorem{question}{Question}

\begin{document}

\title{The Chebotarev invariant of a finite group}
 
\author{Emmanuel  Kowalski}
\address{ETH Zürich -- D-MATH\\
  Rämistrasse 101\\
  8092 Zürich\\
  Switzerland} 
\email{kowalski@math.ethz.ch}

\author{David Zywina}
\address{Department of Mathematics, University of Pennsylvania\\
  Philadelphia, PA 19104-6395\\
  USA} 
\email{zywina@math.upenn.edu}



\begin{abstract}
  We consider invariants of a finite group $G$ related to the number
  of random (independent, uniformly distributed) conjugacy classes
  which are required to generate it. These invariants are intuitively
  related to problems of Galois theory. We find group-theoretic
  expressions for them and investigate their values both theoretically
  and numerically. 
\end{abstract}

\maketitle

\section{Introduction}



A well-known method to compute the Galois group $H$ of a number field
(e.g., of the splitting field of a polynomial $P\in\Zz[T]$ with
integral coefficients) can be described roughly as follows: (1) find a
group $G$ which contains $H$, e.g., because of symmetry considerations
(2) this group $G$ is our ``guess'' for $H$, and we try to prove that
$H=G$ by reducing modulo successive primes, using the fact that the
Frobenius automorphisms give \emph{conjugacy classes} in the Galois
group $H$, and hence conjugacy classes in $G$.
\par
If the guess in (1) was right, and if one is patient enough in (2)
that the conjugacy classes observed are only compatible with the
Galois group being our candidate $G$, then we have succeeded.
\par
This method is particularly simple when $G$ is ``guessed'' to be the
symmetric group acting on the roots of the polynomial: it can lead
quickly to examples of equations with this Galois group. 
In general, however this is not the most efficient algorithm (if only
because the first step (1) is hard to formalize!), and thus computer
algebra systems use other techniques. Still, this method is
well-suited for certain theoretical investigations, for instance for
probabilistic Galois theory (see, e.g.,~\cite{gallagher}), and it can
be surprisingly efficient even for fairly complicated groups (see for
instance our joint works~\cite{jkz} and~\cite{jkz2} with F. Jouve, for
a case where the Weyl group of the exceptional Lie group $E_8$ is the
Galois group involved, and for for further developments along these
lines).
\par
In view of this, it is somewhat surprising that no general study of
the efficiency of the underlying algorithm seems to have been
performed. Among the very few references we know is a paper of
Dixon~\cite{dixon}, who considers informally the case of the symmetric
groups $\sy_n$ and mentions some earlier work of McKay. On the other
hand, there has been a fair amount of interest in the question of
determining the probability that a tuple of elements generate a finite
group, which is the analogue problem where conjugacy is ignored, see
for instance the paper~\cite{kantor-lubotzky} of Kantor and
Lubotzky. The paper~\cite{pomerance} of C. Pomerance considers the
question for abelian groups, when the conjugacy issue is also
irrelevant, and his results do apply to our setting. The current paper
will provide the beginning of the theoretical analysis of this type of
algorithm for general finite groups. As a specific result, we will
prove the following result (informally stated;
Theorem~\ref{th-cheb-symmetric} gives the precise statement using the
definitions of Section~\ref{sec-chebotarev}):

\begin{theorem}[Boundedness of Chebotarev invariants for 
symmetric groups]\label{th-intro}
  There exists a constant $c>0$ with the following property: for all
  integers $n\geq 1$, the average number of independently, randomly
  chosen conjugacy classes\footnote{\ This means distributed in
    proportion with the size of the conjugacy class.} of the symmetric
  group $\sy_n$ one must pick before ensuring that any tuple of
  elements taken from each of these classes generate $\sy_n$, is at
  most $c$. In fact, for any $k\geq 1$, there exists $c_k\geq 0$ such
  that the average of the $k$-th power of this number is bounded by
  $c_k$ for all $n$.
\end{theorem}

Here is the rough outline of this work: we consider probabilistic
models in Section~\ref{sec-chebotarev}, and define an invariant, which
we call the \emph{Chebotarev invariant} of a finite group, using such
a model (the name, based on the Chebotarev density theorem, is
justified in Section~\ref{sec-arith}); it makes precise the informal
notion in the statement of Theorem~\ref{th-intro}. Computing this
invariant is seen to be related to very interesting questions of group
theory, independently of any arithmetic motivation.  In
Section~\ref{sec-ex1}, we indicate how to compute this invariant for
abelian groups (based on Pomerance's work) and in
Section~\ref{sec-solvable} we consider solvable groups of a certain
``extremal'' type. In Sections~\ref{sec-non-ab-2},~\ref{sec-sym-alt}
and~\ref{sec-ex2}, we consider theoretical and numerical examples for
non-abelian, often non-solvable, groups -- in particular alternating
and symmetric groups, proving Theorem~\ref{th-intro}. Finally,
Section~\ref{sec-arith} makes some informal remarks concerning the
applicability of our results to arithmetic situations (the original
motivation); as we will see, there are non-trivial difficulties
involved, and we hope to come back to these questions later.
\par
\medskip
\par
\textbf{Notation.} As usual, $|X|$ denotes the cardinality of a set,
and $\Fp_q$ is a field with $q$ elements. If $G$ is a finite group,
and $H\subset G$, we write
$$
\nu_G(H)=\nu(H)=\frac{|H|}{|G|}.
$$
\par
We write $G^{\sharp}$ for the set of conjugacy classes of $G$, and for
$C\subset G^{\sharp}$, we also write $\nu_G(C)$ or $\nu(C)$ for
$\nu(\tilde{C})$, where $\tilde{C}\subset G$ is the union of all
conjugacy classes in $C$.
\par
In fact, as a matter of convenience, we will usually denote in the
same way a subset of conjugacy classes and the corresponding set of
elements in $G$, unless it is not clear in context if $c\in C$ means
that $c$ is a conjugacy class or an element of $G$ (we will write
often $c^{\sharp}$ for a conjugacy class, avoiding most ambiguity).
\par
We recall that a geometric random variable $X$ with parameter $p\in
[0,1]$ on a probability space is a random variable taking almost
surely values in the set of positive integers, with
\begin{equation}\label{eq-law-geom}
\proba(X=k)=p(1-p)^{k-1}
\end{equation}
for $k\geq 1$. We then have
\begin{equation}\label{eq-expect-geom}
\expect(X)=p\sum_{k\geq 1}{k(1-p)^{k-1}}=\frac{1}{p},\quad\quad
\expect(X^2)=\frac{2-p}{p^2},\quad\quad
\variance(X)=\frac{1-p}{p^2}.
\end{equation}
\par
By $f\ll g$ for $x\in X$, or $f=O(g)$ for $x\in X$, where $X$ is an
arbitrary set on which $f$ is defined, we mean synonymously that there
exists a constant $C\geq 0$ such that $|f(x)|\leq Cg(x)$ for all $x\in
X$. The ``implied constant'' refers to any value of $C$ for which this
holds. It may depend on the set $X$, which is usually specified
explicitly, or clearly determined by the context. Similarly, $f\asymp
g$ means that $f\ll g$ and $g\ll f$ on the same set. On the other hand
$f(x) \sim g(x)$ as $x\ra x_0$ means that $f(x)/g(x)\ra 1$ as $x\ra
x_0$.


\section{The Chebotarev invariant of a finite
  group}\label{sec-chebotarev} 

In this section, we describe a natural probabilistic model for the
recognition algorithm described previously. Fix a finite group $G$.
We first remark that, whereas it does not make sense to say that a
conjugacy class lies in a certain subgroup, unless the latter is a
normal subgroup, it does make sense to say that it lies in a conjugacy
class of subgroups. With that in mind, we define:

\begin{definition}
  Let $G$ be a finite group, and let $C=\{C_1,\ldots,C_m\}\subset
  G^{\sharp}$ be a subset of conjugacy classes in $G$. Then $C$
  \emph{generates} $G$ if, for any choice of representatives $g_i\in
  C_i$ for $1\leq i\leq m$, the elements of the tuple $(g_1,\ldots, g_m)$ generate
  $G$. Equivalently, $C$ generates $G$ if and only if there is no
  (proper) maximal subgroup $H$ of $G$ 
  that has non-empty intersection with each of the $C_i$.
  \footnote{\ Alternately, following~\cite{dixon}, one says that
    elements $(g_1,\ldots,g_m)$ \emph{invariably generate} $G$ if
    their conjugacy classes generate $G$ in the sense above.}
\end{definition}

The equivalence of the two definitions is quite clear
contrapositively: if there are $g_i\in C_i$ which generate a proper
subgroup $H_1$, then each $C_i$ intersects any maximal proper subgroup
$H$ of $G$ that contains $H_1$, and conversely.  Note also that the second condition
can be stated by saying that there is a conjugacy class of maximal
subgroups containing $C$.
\par
The following well-known lemma (due to Jordan) is the basic fact
underlying the whole technique:

\begin{lemma}\label{lm-triv}
Let $G$ be a finite group. Then the set $G^{\sharp}$ of conjugacy
  classes generates $H$. In other words, there is no proper
  subgroup of $G$ which contains a representative from each conjugacy
  class. 
  \end{lemma}

Simple as this is, one should also recall at this point that the
analogue of this lemma is false for infinite groups (even compact
groups, such as $SU(n)$, $n\geq 2$); for further discussion of various
interpretations of this lemma, see~\cite{serre-jordan}.
\par
Now, let $(\Omega,\Sigma,\proba)$ be a fixed probability space with a
sequence $X=(X_n)_{n\geq 1}$ of $G$-valued random variables
$$
X_n\,:\, \Omega\ra G,
$$
and let $X_n^{\sharp}$ be the conjugacy class of $X_n$ in
$G^{\sharp}$: those are $G^{\sharp}$-valued random variables.
\par
Intuitively, those $(X_n^{\sharp})$ are the conjugacy classes that we
see coming ``one by one''; the Chebotarev invariant(s) looks at when
we get enough information to conclude that those conjugacy classes can
not all belong to some proper subgroup of $G$.
\par
We now define a random variable $\tau_{X,G}$ (a \emph{waiting time})
by
\begin{equation}\label{eq-wait-see}
  \tau_{X,G}=\min\{n\geq 1\,\mid\, (X_1^{\sharp},\ldots,X_n^{\sharp})\text{
    generate $G$}\}\in [1,+\infty].
\end{equation}
\par
This depends on the sequence $X=(X_n)$, and it may be always infinite
(e.g., if $X_n=1$ for all $n$!). But it is, in an intuitive sense, the
``finest'' invariant in terms of this probabilistic model.  To obtain
more compact and purely numerical invariants, it is natural to first
take the expectation; this takes values in $[1,+\infty]$.

\begin{definition}
  Let $G$ be a finite group, $X=(X_n)$ a sequence of $G$-valued random
  variables and $\tau_{X,G}$ the waiting time above. The
  \emph{Chebotarev invariant} of $G$ with respect to $X$, denoted
  $\cheb{G;X}$, is the expectation $\cheb{G;X}=\expect(\tau_{X,G})$ of
  this random variable.
\end{definition}

\begin{remark}
We focused on conjugacy classes because this is how applications to
Galois theory are likely to arise, but of course one can similarly
define an invariant using the original random elements $(X_n)$ in
$G$. If $G$ is abelian, the two coincide.
\end{remark}

To have an unambiguously defined invariant, we must use a specific
choice of sequence $(X_n)$. The natural model is that of independent,
uniformly distributed elements in $G$: if $(X_n)$ are independent and
identically uniformly distributed $G$-valued random variables, so that
$$
\proba(X_n=g)=\frac{1}{|G|}\quad\quad\text{ for all } g\in
G,\quad\text{ and all } n\geq 1,
$$
and hence
$$
\proba(X_n^{\sharp}=g^{\sharp})=\frac{|g^{\sharp}|}{|G|},\quad\quad\text{
  for all } g^{\sharp}\in G^{\sharp},\quad\text{ and all } n\geq 1,
$$
then we call $\cheb{G;X}$ \emph{the} Chebotarev invariant, and we just
write $\cheb{G}$.

\begin{remark}
  It may be of interest, at least for numerical purposes, to use a
  sequence $(X_n)$ which is not independent, but is obtained, for
  instance, by a rapidly mixing random walk on $G$. Also, the
  arithmetic analogue for computing Galois groups may be interpreted
  as involving non-independent and non-uniform choices of conjugacy
  classes (see Section~\ref{sec-arith}).
\end{remark}

Other numerical invariants may of course be derived from $\tau_{X,G}$,
starting from the higher moments $\expect(\tau_{X,G}^k)$ for $k\geq
1$. Since it is probabilistically most important, when the expectation
of a random variable is known, to also have a control of its second
moment, we define formally the \emph{secondary Chebotarev invariant}:

\begin{definition}
  Let $G$ be a finite group, $X=(X_n)$ a sequence of $G$-valued random
  variables, and let $\tau_{X,G}$ be the waiting time above. The
  \emph{secondary Chebotarev invariant} of $G$ with respect to $X$, is
  the second moment $\scheb{G;X}=\expect(\tau_{X,G}^2)$. If $(X_n)$ is
  a sequence of independent, uniformly distributed random variables,
  then we write $\scheb{G}$ and call it \emph{the} secondary
  Chebotarev invariant.
\end{definition}

We will now give formulas for the two Chebotarev invariants (in the
independent case), which are expressed purely in terms of
group-theoretic information.  This is useful for explicit
computations, at least for groups which are very well understood (but
often the probabilistic origin of $\cheb{G}$ should also be kept in
mind.)
\par
To state the formulas, we must introduce the following data and
notation about $G$. Let $\max(G)$ be the set of conjugacy classes of
(proper) maximal subgroups of $G$ (if $G$ is trivial, this is empty);
for a conjugacy class of of maximal subgroups $\Hc\in \max(G)$, let
$\Hc^{\sharp}$ denote the set of conjugacy classes $C$ of $G$ which
``occur in $\Hc$'', i.e., such that $C\cap H_1\not=\emptyset$ for
\emph{some} $H_1$ in the conjugacy class $\Hc$.\footnote{\ Note that
  this depends on the underlying group $G$.} Moreover, if $I\subset
\max(G)$ is a set of conjugacy classes of maximal subgroups, we let
\begin{equation}\label{eq-hi}
\Hc_I^{\sharp}=\bigcap_{\Hc\in I}{\Hc^{\sharp}},
\end{equation}
the set of conjugacy classes of $G$ which appear in all subgroups in $I$.
\par
Then we have:

\begin{proposition}\label{pr-comput}
  Let $G$ be a \emph{non-trivial} finite group. With notation as
  above, we have
  \begin{equation}\label{eq-comput}
\cheb{G}=\sum_{\stacksum{I\subset \max(G)}{I\not=\emptyset}}{
\frac{(-1)^{|I|+1}}{
1-\nu(\Hc_I^{\sharp})}},
\end{equation}
and
\begin{equation}\label{eq-secondary}
\scheb{G}=
\sum_{\stacksum{I\subset \max(G)}{I\not=\emptyset}}{
\frac{(-1)^{|I|}}{
1-\nu(\Hc_I^{\sharp})}
\Bigl(
1-\frac{2}{1-\nu(\Hc_I^{\sharp})}
\Bigr)}=
\sum_{\stacksum{I\subset \max(G)}{I\not=\emptyset}}{
(-1)^{|I|+1}\frac{1+\nu(\Hc_I^{\sharp})}{
(1-\nu(\Hc_I^{\sharp}))^2}}.
\end{equation}
\end{proposition}

These formulas do not apply for the trivial group, since they lead to
empty sums which are zero, whereas the definition leads to\footnote{\
  One may argue that there is no need to look at any elements to be
  sure of generating the trivial group, but this does not correspond
  to the definition.}  $\cheb{1}=1$, $\scheb{1}=1$.
\par
Probabilists will have noticed that the first formula (at least) is
very similar to the one for the expectation of the waiting time for a
general coupon collector problem. There is indeed a link, which is
provided by the next lemma where independence of the random elements
$X_n$ is not required.

\begin{lemma}\label{lm-link}
  Let $G$ be a non-trivial finite group and $X=(X_n)$ a sequence of
  $G$-valued random variables. The waiting time $\tau_{X,G}$ is equal
  to
$$
\tau_{X,G}=\max_{\Hc\in \max{G}}{\hat{\tau}_{\Hc}},
$$
where 
\begin{equation}\label{eq-hat}
\hat{\tau}_{\Hc}=\min\{n\geq 1\,\mid\, X_n^{\sharp}\notin \Hc^{\sharp}\}\ ;
\end{equation}
note that $\hat{\tau}_{\Hc}$ depends also on $X$.
\end{lemma}

In other words, $\tau_{X,G}$ is also the maximal $n$ such that we need
to look at $X_i$ for $i$ up to $n$, before we witness, for every
conjugacy class $\Hc$ of maximal subgroups, some $X_n$ which is
incompatible with the groups in this class $\Hc$. This is very close
to a coupon collector problem (see, e.g.,~\cite{flajolet} for a
general description of this type of problems), where the ``coupons''
we need to collect correspond to the conjugacy classes which are not
in $\Hc^{\sharp}$, as $\Hc$ ranges over $\max(G)$.  But since a single
$X_n$ may serve as coupon for more than one $\Hc^{\sharp}$, this does
not exactly correspond to standard coupon collector
problems.\footnote{\ This has been called the ``coupon subset
  collection problem'' by Adler and Ross~\cite{adler-ross}.}  Because
of this, we state and prove the following general abstract result,
which may have other applications.

\begin{proposition}\label{pr-collect}
  Let $(\Omega,\Sigma,\proba)$ be a probability space, $D$ a finite
  set. Let $(Z_n)$ be a sequence of $D$-valued random variables. Let
  $\mathcal{E}$ be a non-empty finite collection of non-empty subsets
  of $D$, and let
$$
\tau_{\mathcal{E}}=\min\{n\geq 1\,\mid\, \text{ for all $E\in
  \mathcal{E}$, there exists some $m\leq n$ with $Z_m\in E$}\}
$$
be the waiting time before all subsets $E\in\mathcal{E}$ have been
witnessed in the sequence $(Z_n)$. For $I\subset \mathcal{E}$,
non-empty, let
\begin{equation}\label{eq-wait-geom}
  T_I=\min\{n\geq 1\,\mid\, Z_n\in E\text{ for some subset } E\in I\}.
\end{equation}
\par
\emph{(1)} Assume $T_I<+\infty$ almost surely for all non-empty
subsets $I\subset \mathcal{E}$. Then we have
\begin{equation}\label{eq-incl-excl}
\tau_{\mathcal{E}}=\sum_{\emptyset\not=I\subset
  \mathcal{E}}{(-1)^{|I|+1}T_{I}}.
\end{equation}
\par
\emph{(2)} Assume the $Z_n$ are independent and identically
distributed random variables and let $\mu$ be their common law. We
have
\begin{equation}\label{eq-expect-collect}
\expect(\tau_{\mathcal{E}})=
\sum_{\stacksum{I\subset \mathcal{E}}{I\not=\emptyset}}{
\frac{(-1)^{|I|+1}}{
\proba(Z_n\in \bigcup_{E\in I}{E})}}=
\sum_{\stacksum{I\subset \mathcal{E}}{I\not=\emptyset}}{
\frac{(-1)^{|I|+1}}{
\mu(\bigcup_{E\in I}{E})}},
\end{equation}
and if the subsets in $\mathcal{E}$ are disjoint, we have
\begin{equation}\label{eq-integrale}
\expect(\tau_{\mathcal{E}})=
\int_0^{+\infty}{
\Bigl(1-\prod_{E\in\mathcal{E}}{(1-\exp(-\mu(E)t))}\Bigr)dt
}.
\end{equation}
\par
\emph{(3)} We have
\begin{equation}\label{eq-sec-collect}
\expect(\tau_{\mathcal{E}}^2)=
\sum_{\stacksum{I\subset \mathcal{E}}{I\not=\emptyset}}{
\frac{(-1)^{|I|}}{
\mu(\bigcup_{E\in I}{E})}\Bigl(
1-\frac{2}{\mu(\bigcup_{E\in I}{E})}
\Bigr)}.
\end{equation}
\end{proposition}

When $\mathcal{E}$ is the set of singletons in $D$, where we have
exactly the coupon collector problem, the formulas for the expectation
are well-known (see, e.g.,~\cite[Theorem 4.1]{flajolet} for the
integral formula); we have not seen general formulas for the second
moment in the literature.

\begin{proof}[Proof of Proposition~\ref{pr-collect}]
To simplify notation, define
\begin{equation}\label{eq-ei}
E_I=\bigcup_{E\in I}{E},
\end{equation}
for each $I\subset \mathcal{E}$. Formula~(\ref{eq-incl-excl})
-- which implies in particular that $\tau_{\mathcal{E}}$ is finite
almost surely -- can be checked, e.g., by seeing $\tau_{\mathcal{E}}$ as
the length of the subset
$$
\bigcup_{E\in \mathcal{E}}{[0,T_{\{E\}}]}\subset \Rr,
$$
(which is therefore almost surely finite by assumption on the
$T_I$), and applying the inclusion-exclusion formula for the
measure of a union of finitely many sets (for the Lebesgue measure, or
the counting measure on $\Zz$, equivalently):
$$
\Bigl|\bigcup_{E\in \mathcal{E}}{[0,T_{\{E\}}]}\Bigr|=
\sum_{\emptyset\not=I\subset \mathcal{E}}{
(-1)^{|I|+1}\Bigl|
\bigcap_{E\in I}{[0,T_{\{E\}}]}
\Bigr|},
$$
at which point it is enough to observe that, for any
$I\subset \mathcal{E}$, we have
$$
\Bigl|
\bigcap_{E\in I}{[0,T_{\{E\}}]}
\Bigr|=
\min_{E\in I}{T_{\{E\}}}=T_I.
$$
\par
We can now finish the computation of $\expect(\tau_{\mathcal{E}})$ in
(2), in the case of independent random variables.  Indeed, in that
case the random variable $T_I$ is distributed like a geometric
random variable with parameter $p=\proba(Z_n\in E_I)$
(see~(\ref{eq-law-geom})) for any non-empty subset $I\subset
\mathcal{E}$, so that taking expectation in~(\ref{eq-incl-excl}) and
applying~(\ref{eq-expect-geom}), we obtain the result.
\par
The integral expression~(\ref{eq-integrale}) is a consequence
of~(\ref{eq-expect-collect}) and the additivity of measure for
disjoint sets: it suffices to expand the product and use
$$
\int_0^{+\infty}{e^{-at}dt}=\frac{1}{a},\quad\text{ for } a>0.
$$
\par
Finally, to compute the second moment in the independent case, we
start with the same formula~(\ref{eq-incl-excl}) to get
$$
\expect(\tau_{\mathcal{E}}^2)=\dblsum_{
\stacksum{\emptyset\not=I\subset \mathcal{E}}
{\emptyset\not=J\subset \mathcal{E}}}{
}{(-1)^{|I|+|J|}\expect(T_{I}T_{J})}.
$$
\par
We first transform this by applying~(\ref{eq-prod-geom}) in
Lemma~\ref{lm-prod-geom} below to compute
$\expect(T_IT_J)$. This gives
\begin{align}
\expect(\tau_{\mathcal{E}}^2)&=\dblsum_{
\stacksum{\emptyset\not=I\subset \mathcal{E}}
{\emptyset\not=J\subset \mathcal{E}}}{
\frac{(-1)^{|I|+|J|}}
{\mu(E_{I\cup J})}
\Bigl\{
\frac{1}{\mu(E_I)}+
\frac{1}{\mu(E_J)}-1
\Bigr\}}\nonumber\\
&=
\dblsum_{
\stacksum{\emptyset\not=I\subset \mathcal{E}}
{\emptyset\not=J\subset \mathcal{E}}}{
\frac{(-1)^{|I|+|J|}}
{\mu(E_{I\cup J})}
\Bigl\{
\frac{2}{\mu(E_I)}-1
\Bigr\}}\quad\quad\text{(by symmetry).}\label{eq-red-w}
\end{align}
\par
To continue, consider more generally arbitrary complex coefficients
$\beta(I)$ defined for $I\subset \mathcal{E}$, and the expression
$$
W(\beta)=\dblsum_{
\stacksum{\emptyset\not=I\subset \mathcal{E}}
{\emptyset\not=J\subset \mathcal{E}}}{
\frac{(-1)^{|I|+|J|}}
{\mu(E_{I\cup J})}\beta(I)};
$$
note that $\expect(\tau_{\mathcal{E}}^2)$ is a simple combination of
two such expressions.
\par
We proceed to reduce $W(\beta)$ to a single sum over $I\subset
\mathcal{E}$ by rearranging the sum according to the value of $I\cup
J$:
$$
W(\beta)=\sum_{\emptyset\not=K\subset \mathcal{E}}{
\frac{1}{\mu(E_K)}
\dblsum_{\stacksum{\emptyset\not=I,J\subset \mathcal{E}}
{I\cup J=K}}
{(-1)^{|I|+|J|}\beta(I)}}.
$$
\par
The inner sum is rearranged in turn as
\begin{align*}
\dblsum_{\stacksum{\emptyset\not=I,J\subset \mathcal{E}}
{I\cup J=K}}
{(-1)^{|I|+|J|}\beta(I)}&=
\sum_{\emptyset\not=I\subset K}{
(-1)^{|I|}\beta(I)
\sum_{\stacksum{\emptyset\not=J\subset K}{I\cup J=K}}
{(-1)^{|J|}}
}\\
&=\sum_{\emptyset\not=I\subset K}{
(-1)^{|I|+|K-I|}\beta(I)
\sum_{\stacksum{I'\subset I}{I'\cup (K-I)\not=\emptyset}}
{(-1)^{|I'|}}
}
\end{align*}
since the subsets $J$ with $I\cup J=K$ are parametrized by $I'\subset
I$ using the correspondence $I'\mapsto (K-I)\cup I'$ with inverse
$J\mapsto J\cap I$.
\par
For fixed $I$, the last summation condition $I'\cup
(K-I)\not=\emptyset$ is always valid, \emph{unless} $I=K$. In that
last case, it only excludes the set $I'=\emptyset$ from all $I'\subset
I$. Since we have, for any finite set $X$, the binomial relation
$$
\sum_{Y\subset X}{(-1)^{|Y|}}=0,
$$
it follows that the double sum is simply given by
$$
\dblsum_{\stacksum{\emptyset\not=I,J\subset \mathcal{E}}
{I\cup J=K}}
{(-1)^{|I|+|J|}\beta(I)}=(-1)^{|K|+1}\beta(K),
$$
and hence
$$
W(\beta)=\sum_{\emptyset\not=K\subset \mathcal{E}}{
\frac{(-1)^{|K|+1}\beta(K)}{\mu(E_K)}}.
$$
\par
Applied to the expression~(\ref{eq-red-w}), this leads precisely
to~(\ref{eq-sec-collect}). 
\end{proof}

To deduce Proposition~\ref{pr-comput}, we apply this proposition with
$$
Z_n=X_n^{\sharp},\quad\quad
D=G^{\sharp}\quad\quad \mathcal{E}=\{G^{\sharp}-\Hc^{\sharp}\,\mid\,
\Hc\in\max(G)\},
$$
in the case where the $(X_n)$ are independent uniformly distributed on
$G$, so that the common distribution is $\mu=\nu$. Since, for
$I\subset \max{G}$, we have
$$
\nu\Bigl(\bigcup_{\Hc\in I}{(G^{\sharp}-\Hc^{\sharp})}\Bigr)=
1-\nu\Bigl(\bigcap_{\Hc\in I}{\Hc^{\sharp}}\Bigr),
$$
the formulas~(\ref{eq-expect-collect}) and~(\ref{eq-sec-collect}) give
exactly the claimed formulas~(\ref{eq-comput})
and~(\ref{eq-secondary}).

\begin{remark}
  Note the following strange-looking ``linearity'' property, which can
  be checked from our formulas and~(\ref{eq-prod-geom}): for
  $G\not=1$, we have
\begin{gather*}
\cheb{G}=\expect(\tau_G)=\sum_{\emptyset\not=I\subset \max{G}}{
(-1)^{|I|+1}\expect(T_I)
},\\
\scheb{G}=\expect(\tau_G^2)=\sum_{\emptyset\not=I\subset \max{G}}{
(-1)^{|I|+1}\expect(T_I^2)
}.
\end{gather*}
\end{remark}

\begin{remark}
  These formulas can only be useful for practical computation if the
  number of conjugacy classes of maximal subgroups of $G$ is fairly
  small, or if they are very well understood.  As a theoretical
  instrument, they suffer from the fact that it is very hard to use
  them to guess or estimate the actual value of $\cheb{G}$. For
  instance, it is not clear how to recover even the trivial lower
  bound
\begin{equation}\label{eq-triv-lb}
\cheb{G}\geq \delta(G),
\end{equation}
where $\delta(G)$ is the minimal cardinality of a generating set of
$G$. We will give examples later on where this bound is very close to
the truth, even with groups of size growing to infinity.
\end{remark}

\begin{remark}
Another natural formula is
\begin{equation}\label{eq-formula-probas}
\cheb{G}=1+\sum_{n\geq 1}{\proba(X_1^{\sharp},\ldots,X_{n}^{\sharp}\text{
    do not generate } G)},
\end{equation}
which is also valid for $G=1$.
\par
Indeed, since $\tau_G$ takes positive integer values, we have  the
familiar formula
$$
\expect(\tau_G)=\sum_{n\geq 1}{\proba(\tau_G\geq n)},
$$
and clearly
$$
\{\tau_G\geq n\} = \{X_1^{\sharp},\ldots,X_{n-1}^{\sharp}\text{ do not
  generate } G\},
$$
for $n\geq 1$. When $n=1$, this is the certain event, with probability
one, thus leading to~(\ref{eq-formula-probas}).
\par
A moment's thought shows that one can also identify this formula with
the one coming from~(\ref{eq-comput}) by expanding the geometric
series: 
\begin{align*}
\cheb{G}=\sum_{\stacksum{I\subset \max(G)}{I\not=\emptyset}}{
\frac{(-1)^{|I|+1}}{
1-\nu(\bigcap_{\Hc\in I}{\Hc^{\sharp}})}}
&=
\sum_{\stacksum{I\subset \max(G)}{I\not=\emptyset}}{
(-1)^{|I|+1}\sum_{n\geq 0}{\nu(\bigcap_{\Hc\in I}{\Hc^{\sharp}})^n}}\\
&=1+\sum_{n\geq 1}{\Bigl(
\sum_{\stacksum{I\subset \max(G)}{I\not=\emptyset}}{
(-1)^{|I|+1}\nu(\bigcap_{\Hc\in I}{\Hc^{\sharp}})^n}
\Bigr)},
\end{align*}
(where the term with $n=0$ is only equal to $1$ for
$\max(G)\not=\emptyset$, i.e., $G\not=1$).
\par
This is not really a different proof of Proposition~\ref{pr-comput}
since the relation
$$
\proba(X_1^{\sharp},\ldots,X_{n}^{\sharp}\text{
do not generate } G)=
\sum_{\stacksum{I\subset \max(G)}{I\not=\emptyset}}{
(-1)^{|I|+1}\nu(\bigcap_{\Hc\in I}{\Hc^{\sharp}})^n}
$$
is proved by inclusion-exclusion, exactly as in the proof of
Proposition~\ref{pr-collect}. 
\par
The point of~(\ref{eq-formula-probas}) is rather that it leads to
another lower bound (for $G\not=1$):
$$
\cheb{G}\geq 1+\sum_{n\leq
  k}{\proba(X_1^{\sharp},\ldots,X_{n}^{\sharp}\text{ do not generate }
  G)}
$$
for any fixed $k\geq 1$, and this may be quite useful because there
has been a large amount of work on the estimation of those
probabilities when $k$ is small, e.g., $k=2$ if $G$ is not cyclic. For
instance, Dixon (for $G=A_n$ and $n\ra +\infty$) and Kantor and
Lubotzky (for $G=\mathbf{G}(\Fp_q)$ a simple classical group and $q\ra
+\infty$) have shown that in those cases we have
$$
\proba(X_1,X_2\text{ do not generate } G)\ra 0
$$
as $n$ (resp. $q$) goes to infinity, indeed with quantitative
estimates (see~\cite{kantor-lubotzky}) -- but note the probabilities
with elements and with conjugacy classes may behave rather differently
(e.g., for $G=PSL(2,\Fp_p)$, the probability that two random elements
generate $G$ is very close to $1$ for large $p$, but there is a
probability converging to $1/2$ that two random conjugacy classes do
not generate $G$, see the proof of Theorem~\ref{th-psl}). In a
sense, the Chebotarev invariant is thus a refinement of these type of
probabilities. We refer to~\cite{dixon2} for a brief survey of
probabilistic Galois theory, and to~\cite{dixon} for the analysis of
the case of symmetric groups.
\end{remark}

\begin{remark}\label{rm-maximal}
As explained in~\cite[Th. 5]{serre-jordan}, we have
\begin{equation}\label{eq-maximal}
\nu(\Hc^{\sharp})\leq 1-\frac{1}{|G/H|}
\end{equation}
for any conjugacy class of maximal subgroup of $G$ (this is due to
Cameron and Cohen).
\end{remark}

We now prove the lemma which supplies the formula~(\ref{eq-prod-geom})
used in the proof of the proposition.

\begin{lemma}
\label{lm-prod-geom}
With notation as in Proposition~\ref{pr-collect}, in particular with
independent identically distributed random variables $(Z_n)$, for any
two non-empty subsets $I$, $J$ in $\mathcal{E}$, we have
\begin{equation}\label{eq-prod-geom}
\expect(T_IT_J)=\frac{1}{\mu(E_{I\cup J})}\Bigl(
\frac{1}{\mu(E_I)}+\frac{1}{\mu(E_J)}-1\Bigr).
\end{equation}
\end{lemma}

\begin{proof}
  This is a fairly direct computation, but not very enlightening (at
  least in our presentation; there might be other approaches that
  makes this more transparent).
\par
We first compute the joint distribution of $T_I$ and $T_J$, and
for this, we use the shorthand notation
\begin{gather*}
  p=\mu(E_I),\quad\quad q=\mu(E_J),\quad\quad r=\mu(E_{I\cap
    J}),\quad\quad
  s=\mu(E_{I\cup J}),\\
  p'=\mu(E_I-E_J),\quad\quad q'=\mu(E_J-E_I).
\end{gather*}
\par
Then we have (generalizing the geometric distribution of a single
$T_I$): 
\begin{equation}\label{eq-joint-law}
\proba(T_I=n\text{ and } T_J=m)=
\begin{cases}
(1-s)^{n-1}r,&\text{ if } n=m\geq 1,\\
(1-s)^{m-1}q'(1-p)^{n-m-1}p,&\text{ if } n>m\geq 1,\\
(1-s)^{n-1}p'(1-q)^{m-n-1}q,&\text{ if } m>n\geq 1.
\end{cases}
\end{equation}
\par
To justify, e.g., the second of these, note that $T_I=n>m=T_J$
means that $Z_k$ must not be in $E_{I\cup J}$ for $k\leq m-1$, $Z_m$
must be in $E_J$ but not in $E_I$, $Z_k$ must not be in $E_I$ for
$m<k<n$, and finally $Z_n$ must be in $E_I$; then the independence of
the $(Z_n)$ gives the formula.
\par
Now we write
$$
\expect(T_IT_J)=\dblsum_{n, m\geq 1}{nm\proba(T_I=n\text{ and }
  T_J=m)}, 
$$
and we split the sum according to the three cases, say
$$
\expect(T_IT_J)=Q_1+Q_2+Q_3.
$$
\par
Introducing further the functions
$$
\varphi_{i}(x)=\sum_{n\geq 1}{n^i(1-x)^{n-1}},
$$
we have the expressions
$$
Q_1=\sum_{n\geq 1}{n^2\proba(T_I=T_J=n)}=
r\sum_{n\geq 1}{n^2(1-s)^{n-1}}=r\varphi_2(s),
$$
and 
\begin{align*}
Q_2&=\sum_{1\leq m<n}{nm\proba(T_I=n\text{ and }T_J=m)}\\
&=q'p\sum_{1\leq m<n}{nm(1-s)^{m-1}(1-p)^{n-m-1}}\\
&=q'p\sum_{m\geq 1}{m(1-s)^{m-1}\sum_{k\geq 1}{(m+k)(1-p)^{k-1}}}\\
&=q'p(\varphi_0(p)\varphi_2(s)+\varphi_1(p)\varphi_1(s)),
\end{align*}
while $Q_3$ is given by the same expression after exchanging $p$ and
$q$, $p'$ and $q'$.
\par
Since, by Taylor expansion, we have
$$
\varphi_0(x)=\frac{1}{x},\quad
\varphi_1(x)=\frac{1}{x^2},\quad
\varphi_2(x)=\frac{2-x}{x^3}=\frac{2}{x^3}-\frac{1}{x^2},
$$
we obtain
$$
\expect(T_IT_J)=\frac{1}{s}\Bigl(\frac{2}{s}+
\frac{q'}{ps}+\frac{p'}{qs}-1\Bigr), 
$$
by adding the three terms. Finally, the relations
$$
q'+p=s,\quad\quad p'+q=s,
$$
lead to the simplified expression
\begin{align*}
  \expect(T_IT_J)&=\frac{1}{s}\Bigl(\frac{2}{s}+
  \frac{q'}{ps}+\frac{p'}{qs}-1\Bigr)=
  \frac{1}{s}\Bigl(\frac{2}{s}+\frac{s-p}{ps}+\frac{s-q}{qs}-1\Bigr)\\
  &=\frac{1}{s}\Bigl(\frac{2}{s}+
\frac{1}{p}-\frac{1}{s}+\frac{1}{q}-\frac{1}{s}-1\Bigr),
\end{align*}
which gives~(\ref{eq-prod-geom}).
\end{proof}

\par
\medskip
\par
We now present some easy formal properties of the Chebotarev
invariants (attached, unless stated otherwise, with a sequence of
independent uniformly distributed random variables).
\par
The first lemma may be used to simplify and expand the range of groups
covered by certain computations (this is also observed by
Pomerance~\cite{pomerance}, for the case of numbers of generators
instead of conjugacy classes):

\begin{lemma}\label{lm-frattini}
  Let $G$ be a finite group, and let $\Phi(G)$ be the \emph{Frattini
    subgroup} of $G$, i.e., the intersection of all maximal subgroups
  of $G$. Then, for any normal subgroup $N\triangleleft G$ such that
  $N\subset \Phi(G)$, in particular for $N=\Phi(G)$, we have
$$
\cheb{G}=\cheb{G/N},\quad\quad
\scheb{G}=\scheb{G/N}.
$$
\end{lemma}

\begin{proof}
  Let $H=G/N$. Moreover, we have $\Phi(H)=\Phi(G)/N$ and hence
  $H/\Phi(H)\simeq G/\Phi(G)$. This means that we need only prove the
  result when $N=\Phi(G)$, the general case following by applying this
  to $H$.
\par
Let $\pi\,:\, G\ra G/\Phi(G)$ be the quotient map. If $(X_n)$ is a
sequence of independent random variables uniformly distributed on $G$,
the $Y_n=\pi(X_n)$ are independent and uniformly distributed on
$G/\Phi(G)$. Moreover, for any $n\geq 1$, the elements
$(X_1^{\sharp},\ldots,X_n^{\sharp})$ generate $G$ if and only if the
elements $(Y_1^{\sharp},\ldots, Y_n^{\sharp})$ generate $G/\Phi(G)$:
indeed, this follows from the basic fact that a subset $S\subset G$
generates $G$ if and only if $\pi(S)$ generates $G/\Phi(G)$ (this is
applied to all sets $S=\{x_1,\ldots, x_n\}$ where $x_i$ conjugate to
$X_i$). This gives the result immediately from the definition of the
waiting times.
\end{proof}
\par
We next consider products:

\begin{proposition}\label{pr-product}
  Let $G_1$, $G_2$ be finite groups such that the only subgroup
  $H\subset G_1\times G_2$ which surjects by projection to both
  factors is $H=G$. Then we have
$$
\cheb{G_1\times G_2}\leq \cheb{G_1}+\cheb{G_2}-1.
$$
\end{proposition}

Examples of groups $G_1$, $G_2$ satisfying the hypothesis are any pair
of non-isomorphic simple groups; note that this proposition suggests
that sometimes $\chebp{G}=\cheb{G}-1$ would be a more natural
invariant to consider, since we then have the simpler inequality
$$
\chebp{G_1\times G_2}\leq \chebp{G_1}+\chebp{G_2}.
$$

\begin{proof}
  With $G=G_1\times G_2$ and denoting $X_n=(Y_n,Z_n)\in G_1\times G_2$
  a sequence of independent uniformly distributed random variables, it
  is clear that $(Y_n)$, $(Z_n)$ are similarly independent uniformly
  distributed on $G_1$ and $G_2$ respectively. We then have the
  inequality
$$
\tau_G\leq \max(\tau_1,\tau_2)\leq \tau_1+\tau_2-1
$$
(since $\tau_i\geq 1$ and $\max(m,n)\leq n+m-1$ for integers $n$,
$m\geq 1$), with
$$
\tau_1=\min\{n\geq 1: (Y_1^{\sharp},\ldots,Y_n^{\sharp})\text{ generate }
G_1\},\quad
\tau_2=\min\{n\geq 1: (Z_1^{\sharp},\ldots,Z_n^{\sharp})\text{ generate }
G_2\}
$$
which are distributed like $\tau_{G_1}$, $\tau_{G_2}$ (indeed, if
$n\geq \max(\tau_1,\tau_2)$, then the group generated by any elements
in $X_n^{\sharp}=(Y_n^{\sharp},Z_n^{\sharp})$ surjects to $G_1$ and
$G_2$, hence it must be equal to $G$ by assumption).  Taking
expectation, we get the inequality stated.
\end{proof}

The next result gives upper and lower estimates for the Chebotarev
invariant using smaller sets of maximal subgroups than $\max(G)$. This
can be very useful in particular for the asymptotic study of
$\cheb{G_n}$ for a sequence of finite groups $(G_n)$, as we will see
later on.

\begin{proposition}\label{pr-general-estimates}
Let $G$ be a finite group, and let $M\subset   \max(G)$ be an
arbitrary non-empty finite subset of maximal subgroups. Let
$$
\tilde{\tau}_M=\max_{\Hc\in M}{\hat{\tau}_{\Hc}}.
$$
with notation as in~\emph{(\ref{eq-hat})} and
\begin{equation}\label{eq-pm}
  p_M=\nu\Bigl(G^{\sharp}-\bigcup_{\Hc\in \max(G)-M}{\Hc^{\sharp}}\Bigr).
\end{equation}
\par
We then have
$$
\expect(\tilde{\tau}_M)
=\sum_{\emptyset\not=I\subset M}{\frac{(-1)^{|I|+1}}{
1-\nu(\bigcap_{\Hc\in I}{\Hc^{\sharp}})}}
\leq \cheb{G}\leq \expect(\tilde{\tau}_M)-1+p_M^{-1}
$$
and
$$
\expect(\tilde{\tau}_M^2)\leq
\scheb{G}\leq \expect(\tilde{\tau}_M^2)+
\frac{2-p_M}{p_M^2}-1.
$$
\end{proposition}

\begin{proof}
Define the additional waiting time
$$
\tau^*=\min\{n\geq 1
\,\mid\, X_n\notin \bigcup_{\Hc\notin M}{\Hc^{\sharp}}\}.
$$
\par
We then note the inequalities
$$
\tilde{\tau}_M\leq \tau_G\leq \max(\tilde{\tau}_M,\tau^*)\leq
\tilde{\tau}_M+\tau^*-1, 
$$
where the first inequality is obvious, while the second follows
because, for $n=\max(\tilde{\tau}_M,\tau^*)$, we know that the group
generated by $(X_1^{\sharp},\ldots, X_n^{\sharp})$ is not contained in
any subgroup in a conjugacy class of maximal subgroups $\Hc\in M$, and
that this group also contains one element which is not conjugate to
any element in a subgroup not in $M$.
\par
Now we take expectation on both sides; observing that, by
independence, $\tau^*$ is distributed like a geometric random variable
with parameter $p_m$ given by~(\ref{eq-pm}), we obtain the first
inequalities, using Proposition~\ref{pr-collect}
and~(\ref{eq-expect-geom}).
\par
Similarly, for the secondary invariant, we use the inequalities
$$
\tilde{\tau}_M^2\leq \tau_G^2\leq \max(\tilde{\tau}_M,\tau^*)^2\leq
\tilde{\tau}_M^2+(\tau^*)^2-1,
$$
and get
$$
\expect(\hat{\tau}_M^2)\leq \scheb{G}\leq
\expect(\hat{\tau}_M^2)+\expect((\tau^*)^2)-1=
\expect(\hat{\tau}_M^2)+\frac{2-p_M}{p_M^2}-1.
$$
\end{proof}

The following immediately follows:

\begin{corollary}\label{cor-asymptotic}
  Let $(G_n)$ be a sequence of non-trivial finite groups, and let
  $\nu_n$ denote the corresponding density.  For each $n\geq 1$, let
  $M_n$ be a non-empty subset of $\max(G_n)$, and assume that
\begin{equation}\label{eq-cond-asymp}
\lim_{n\ra +\infty}{\nu_n\Bigl(\bigcup_{\Hc\in
 \max(G_n)-M_n}{\Hc^{\sharp}}\Bigr)}=0,
\end{equation}
i.e., the proportion of elements represented by a conjugacy class in
some subgroup in $M_n$ goes to zero. Then we have
$$
\cheb{G_n}=\expect(\tilde{\tau}_{M_n})+o(1),\quad\quad
\text{ and }\quad\quad
\scheb{G_n}=\expect(\tilde{\tau}_{M_n}^2)+o(1),
$$
as $n\ra +\infty$, with notation as in
Proposition~\ref{pr-general-estimates}.
\end{corollary}

The following sections will now take up the problem of computing, or
estimating, the Chebotarev invariants for various classes of groups.

\section{Abelian and nilpotent groups}\label{sec-ex1}

In this section, we look at finite \emph{abelian} and nilpotent groups
$G$. In fact, because nilpotent groups have the (characteristic)
property that $[G,G]\subset \Phi(G)$ (see, e.g.,~\cite[Th. 11.3,
(v)]{rose}), Lemma~\ref{lm-frattini} shows that  if $G$ is a nilpotent
group, we have
$$
\cheb{G}=\cheb{G/[G,G]},\quad\quad
\scheb{G}=\scheb{G/[G,G]}
$$
which are Chebotarev and secondary Chebotarev invariants of abelian
groups. This applies, in particular, to all $p$-groups.
\par
We will not use the formula from Proposition~\ref{pr-comput} (although
it is possible, as was done in a first draft, to do some computations
using it), because in abelian groups there tends to be many maximal
subgroups up to conjugacy -- since conjugacy is now trivial. Following
the work of Pomerance~\cite{pomerance}, who computed $\cheb{G}$ (with
different terminology) for any abelian group $G$, we will use another
description of the Chebotarev waiting time in the case of abelian
groups.

\begin{theorem}[Pomerance]\label{th-pomerance}
  Let $G$ be a finite abelian group, and for any prime number $p\mid
  |G|$, let $r_p(G)=\dim_{\Fp_p}(G/pG)$ be the $p$-rank of $G$. Let
  $\delta(G)=\max r_p(G)$ be the minimal cardinality of a generating
  set of $G$. Then we have
$$
\cheb{G}=\delta(G)+\sum_{j\geq 1}{
\Bigl(
1-\prod_{p\mid |G|}{\prod_{1\leq i\leq r_p(G)}{
(1-p^{-(\delta(G)+j-i)})}}
\Bigr)
}.
$$
\par
In particular, for $G=\Zz/n\Zz$ with $n\geq 2$, we have
\begin{equation}\label{eq-cheb-cyclic}
\cheb{G}=-\sum_{\stacksum{d\mid n}{d\not=1}}{
\frac{\mu(d)}{1-d^{-1}}}
\end{equation}
and for $G=\Fp_p^k$, where $\Fp_p=\Zz/p\Zz$, with $p$ prime and $k\geq
1$, we have
\begin{equation}\label{eq-cheb-vs}
\cheb{G}=k+\sum_{1\leq j\leq k}{\frac{1}{p^j-1}}.
\end{equation}
\end{theorem}

This is~\cite[Theorem]{pomerance} and immediate corollaries of it. The
formula for $G=\Zz/n\Zz$ might be easier to get directly from
Proposition~\ref{pr-comput}. Indeed, the subgroups of $\Zz/n\Zz$ are
the groups $H_d=d\Zz/n\Zz$ for $d\mid n$, with $\nu(H_d)=d^{-1}$ and
$H_d\cap H_e=H_{[d,e]}$, and the maximal subgroups among these
correspond to minimal divisors of $n$ for divisibility, i.e., to the
primes $p$ dividing $n$. Then a non-empty subset $I$ of $\max(G)$ can
be parametrized by the corresponding subset of prime divisors of $n$,
or equivalently by the squarefree divisor $d>1$ of $n$ which is the
product of those primes. In this correspondence, we have
$$
\bigcap_{H\in I}{H}=\bigcap_{p\mid d}{H_p}=H_d,\quad\quad\text{ hence
}
\quad\quad \nu\Bigl(\bigcap_{H\in I}{H}\Bigr)=\nu(H_d)=\frac{1}{d},
$$
and $(-1)^{|I|}=\mu(d)$, hence~(\ref{eq-comput}) gives the stated
formula for $\cheb{\Zz/n\Zz}$.
\par
We have similar results for the secondary Chebotarev invariant;
Pomerance mentions the possibility of computing these, but does not
give any results in his paper.

\begin{theorem}\label{th-secondary-abelian}
Let $G$ be a finite abelian group. With notation as in
Theorem~\ref{th-pomerance}, we have
$$
\scheb{G}=\delta(G)^2+\sum_{j\geq 1}{(2j+2\delta(G)-1)\Bigl(1-\prod_{p\mid |G|}{
\prod_{1\leq i\leq r_p(G)}{
(1-p^{-(\delta(G)+j-i)})}
}\Bigr)}.
$$
\par
In particular, we have
$$
\scheb{\Zz/n\Zz}=-\sum_{2\leq d\mid
  n}{\mu(d)\frac{1+d^{-1}}{(1-d^{-1})^2} }
$$
for $n\geq 1$ and
$$
\scheb{\Fp_p^k}=
\cheb{\Fp_p^k}^2+\sum_{1\leq j\leq k}{
\frac{p^j}{(p^j-1)^2}
},
$$
for $p$ prime and $k\geq 1$.
\end{theorem}

\begin{proof}
The first result is obtained by reasoning as
in~\cite[p. 195]{pomerance}, with $r$ and $(r+j)$ there replaced by
$r^2$ and $(r+j)^2$. The point is that he shows that
$$
\proba((X_1,\ldots,X_{\delta(G)+j})\text{ generate } G)=
\prod_{p\mid |G|}{
\prod_{1\leq i\leq r_p(G)}{
(1-p^{-(\delta(G)-r_p(G)+j+i)})
}
}.
$$
\par
To deduce the values for $G=\Zz/p^k\Zz$, it is simpler to use the
description\footnote{\ Which is the analogue of the decomposition of
  the waiting time for the standard Coupon Collecting Problem in a sum
  of geometric random variables.}
$$
\tau_{G}=\sum_{j=1}^k{G_j}
$$
where the $G_j$ are independent geometric random variables with
parameters $p_j=1-p^{-j}$. Concretely, they can be defined as follows
\begin{gather*}
G_k=\min\{n\geq 1\,\mid\, X_n\not=0\},\\
G_{k-1}=\min\{n\geq 1\,\mid\, \dim_{\Fp_p} \langle
X_{G_k+n},X_{G_k}\rangle=2\},\quad\ldots \\
G_{1}=\min\{n\geq 1\,\mid\, \dim_{\Fp_p} \langle
X_{G_2+n},X_{G_2},\ldots,X_{G_k}\rangle=k\},
\end{gather*}
which, by independence of the $(X_n)$, are easily checked to be
indeed independent geometric variables with the stated parameters.
\par
This decomposition leads to the formula for~$\scheb{G}$ immediately,
using~(\ref{eq-expect-geom}) and additivity of the variance of
independent random variables.
\end{proof}

The formula of Pomerance gives a quick way to understand the limit
values of Chebotarev invariants for abelian groups with a given rank
$\delta(G)$. 

\begin{corollary}[Pomerance]\label{cor-pomerance}
  For any fixed integer $k\geq 1$, and any abelian finite group $G$
  with $\delta(G)=k$, we have
$$
k\leq \cheb{G}\leq \limsup_{\stacksum{|G|\ra
    +\infty}{\delta(G)=k}}{\cheb{G}}= k+1+\sum_{j\geq
  1}{\Bigl(1-\prod_{1\leq j\leq k}{\zeta(j+k)^{-1}}\Bigr)}.
$$
\par
In particular, the Chebotarev invariants for cyclic groups are bounded.
\end{corollary}

\begin{corollary}\label{cor-concentration}
For any fixed $k$, we have
$$
\cheb{\Fp_p^k}=k+O(p^{-1}),\quad\quad
\scheb{\Fp_p^k}=k^2+O(p^{-1}),
$$
and 
$$
\proba(\tau_{\Fp_p^k}\not=k)
\ll p^{-1},
$$
where the implied constants depend only on $k$.
\end{corollary}

This last result shows that, for vector spaces over a finite field,
the Chebotarev invariant is strongly peaked around the average, which
is itself close to the dimension.

\begin{proof}
  Only the last inequality needs (maybe) a bit of explanation.  Since
  $\tau_{\Fp_p^k}$ takes positive integer values $\geq k$, we have
$$
|\tau_{\Fp_p^k}-k|\geq 1
$$
if $\tau_{\Fp_p^k}\not=k$.  Hence,  if $\tau_{\Fp_p^k}\not=k$, we have
$$
|\tau_{\Fp_p^k}-\cheb{\Fp_p^k}|\geq
|\tau_{\Fp_p^k}-k|-|\cheb{\Fp_p^k}-k|\geq 1-|\cheb{\Fp_p^k}-k|,
$$
and if we furthermore we have $p\geq p_0$, where $p_0$ (depending on
$k$) is chosen so that
$$
k\leq \cheb{\Fp_p^k}\leq k+1/2
$$
for all $p\geq p_0$, it follows that
$$
\{\tau_{\Fp_p^k}\not=k\}
\subset \{|\tau_{\Fp_p^k}-\cheb{\Fp_p^k}|\geq 1/2\}
$$
for such $p$, and then the Chebychev inequality gives
$$
\proba(\tau_{\Fp_p^k}\not=k)\leq 4\variance(\tau_{\Fp_p^k})\ll p^{-1}
$$
for $p\geq p_0$, where the implied constant depends on $k$. Increasing
this constant if needed (e.g., taking it to be at least $p_0$), we
can also claim that this inequality holds for $p\geq 2$.
\end{proof}

\begin{remark}
In particular, for cyclic groups, the Chebotarev invariant is at most,
and its limsup is, the constant
$$
2+\sum_{k\geq
  2}{\Bigl(1-\frac{1}{\zeta(k)}\Bigr)}=
2.705211140105367764\ldots
$$
\par
This asymptotic behavior is not without interest (and some surprise):
on the one hand, we see that $\cheb{\Zz/n\Zz}$ remains absolutely
bounded, despite the existence of cyclic groups with many subgroups,
and on the other hand, we see that it is not always close to the
minimal number of generators. 
\par
Concerning the first point, notice for instance that a ``naive''
invariant is given by
$$
\sum_{H\in \max(G)}{\nu(H)}=\sum_{p\mid n}{\frac{1}{p}},
$$
and this is unbounded as $n$ grows for $G=\Zz/n\Zz$ (though it is $\ll
\log\log \log n$, and thus bounded in practice...), see the discussion
surrounding~(\ref{eq-naive}) below for occurrences of such quantities
instead of the Chebotarev invariant.
\par 
For the second, note that (interpreting $1/\zeta(1)=0$) the limsup we
found is also $N$ where
$$
N=\sum_{k\geq 1}{(1-\zeta(k)^{-1})}
$$
is sometimes called the \emph{Niven constant}. Niven obtained it as
the mean-value of the maximal exponent of a prime dividing a positive
integer:
$$
N=\lim_{n\ra +\infty}{\frac{1}{n-1}\sum_{2\leq j\leq n}{\alpha(j)}}
$$
with
$$
\alpha(j)=\max\{\nu\geq 0\,\mid\, p^{\nu}\mid j\text{ for some prime }
p\}
$$
for $j\geq 2$ (see~\cite{niven}). The explanation for this coincidence
is that $\zeta(k)^{-1}$, for $k\geq 2$, is both the (asymptotic)
density of primitive vectors in $\Zz^k$ and that of $k$-power-free
integers.
\end{remark}

\begin{remark}
If one uses Proposition~\ref{pr-comput} instead, one can prove (after
some computation) the following formulas
\begin{gather*}
\cheb{\Fp_p^k}=
\sum_{1\leq j\leq k}{\frac{(-1)^{j+1}}{1-p^{-j}}\binom{k}{j}_pp^{j(j-1)/2}},\\
\scheb{\Fp_p^k}=\sum_{1\leq j\leq k}{(-1)^j
\frac{1+p^{-j}}{(1-p^{-j})^2}
\binom{k}{j}_pp^{j(j-1)/2}
},
\end{gather*}
where 
$$
\binom{k}{j}_p=\frac{(1-p^{k})\cdots (1-p^{k-j+1})}{(1-p^j)\cdots
(1-p)}
$$
are the $p$-binomial coefficients. Note that those formulas do not
immediately reveal the limiting behavior as $p\ra +\infty$, since the
summands have different degrees as rational functions of $p$.
\end{remark}

\section{A solvable example}\label{sec-solvable}

The results of the previous section, as well as those we will see in
the next one, reveal (or suggest) rather small values of the
Chebotarev invariants, in comparison with the size of the groups. The
following example in the solvable case exhibits very different
behavior (possibly the worse possible).

\begin{proposition}\label{pr-axplusb}
For $q$ a power of a prime, let
$$
H_{q}=\Bigl\{
\begin{pmatrix}
a&t\\0&1\end{pmatrix}\,\mid\, a\in\Fp_q^{\times},
\quad t\in\Fp_q
\Bigr\}
$$
be the group of translations and dilations of the affine plane
$\Fp_q^2$ of order $q(q-1)$, isomorphic to a semi-direct product
$\Fp_q\rtimes \Fp_q^{\times}$.
\par
\emph{(1)} We have
\begin{gather}\label{eq-caxpb}
\cheb{H_q}=
q-q^{-1}\sum_{1\not=d\mid
  q-1}{\frac{\mu(d)}{(1-d^{-1})(1-d^{-1}+q^{-1})}}
\\
\scheb{H_q}=
q(2q-1)+
\scheb{\Zz/(q-1)\Zz}
+\sum_{1\not=d\mid q-1}{\mu(d)\frac{1+d^{-1}-q^{-1}}{
(1-d^{-1}+q^{-1})^2}}.
\label{eq-saxpb}
\end{gather}
\par
\emph{(2)} For $q\geq 2$, we have
\begin{equation}\label{eq-asymp-axpb}
\cheb{H_q}=q+O(\tau(q-1)),\quad\quad
\scheb{H_q}= q(2q-1)+O(\tau(q-1)),
\end{equation}
where $\tau(q-1)$ is the number of positive divisors of $q-1$. In
particular, $\cheb{H_q}\sim q$ as $q\ra +\infty$.
\end{proposition}

Since we have a split exact sequence
$$
1\ra \Fp_q\ra H_q\fleche{\det}\Fp_q^{\times}\ra 1
$$
and the two surrounding groups are isomorphic to $\Fp_p^k$, where
$q=p^k$ with $p$ prime, and to a cyclic group $\Zz/(q-1)\Zz$, with
Chebotarev invariants tending to $k$ as $p$ gets large, and bounded,
respectively, this shows in particular that the Chebotarev invariant
can jump quite uncontrollably under extensions.
\par
The proof will use Proposition~\ref{pr-comput}. We start with a lemma
that is certainly well-known, but for which we give a proof for
completeness and lack of a suitable reference.

\begin{lemma}\label{lm-max-hq}
\emph{(1)} There are $q$ conjugacy classes in $H_q$; they are given,
with representatives of them, by
$$
g_b=\begin{pmatrix}b&0\\0&1
\end{pmatrix},\quad g_b^{\sharp}=\{g\in H_q\,\mid\, \det(g)=b\},
\quad
|g_b^{\sharp}|=q,
$$
where $b\in\Fp_q^{\times}-\{1\}$, and
\begin{gather*}
\mathrm{Id}=\begin{pmatrix}1&0\\0&1
\end{pmatrix},\quad \mathrm{Id}^{\sharp}=\{\mathrm{Id}\},\\
u=\begin{pmatrix}1&1\\0&1
\end{pmatrix},\quad u^{\sharp}=\{
g\in H_q-\{\mathrm{Id}\}\,\mid\, \det(g)=1
\},\quad |u^{\sharp}|=q-1.
\end{gather*}
\par
\emph{(2)} The conjugcay classes of maximal subgroups of $H_q$ have
representatives given by
$$
A=\Bigl\{
\begin{pmatrix}
a&0\\0&1\end{pmatrix}\,\mid\, a\in\Fp_q^{\times}
\Bigr\},
$$
and
$$
C_{\ell}=\Bigl\{\begin{pmatrix}a&t\\0&1
\end{pmatrix}
\in H_q\,\mid\, a\in (\Fp_{q}^{\times})^{\ell}\text{ and } t\in\Fp_q\Bigr\},
$$
where $\ell$ runs over the prime divisors of $q-1$. 
\end{lemma}

\begin{proof}
(1) We have the general conjugation formula
\begin{equation}\label{eq-hq-conjugate}
\begin{pmatrix}a&t\\0&1
\end{pmatrix}
\cdot
\begin{pmatrix}b&v\\0&1
\end{pmatrix}
\cdot
\begin{pmatrix}a&t\\0&1
\end{pmatrix}^{-1}=
\begin{pmatrix}b&av+t(1-b)\\0&1
\end{pmatrix}
\end{equation}
from which it immediately follows that all elements with
$\det(g)=b\not=1$ are conjugate, and gives therefore the $q-2$
conjugacy classes with representatives $g_{b}$ described. For $b=1$,
it is clear that all elements with $b=1$, $v\not=0$ form the conjugacy
class $u^{\sharp}$, and only the identity class remains.
\par
(2) Denote
$$
U=H_q\cap SL(2,\Fp_q)=
\Bigl\{\begin{pmatrix}
1&t\\0&1
\end{pmatrix}
\,\mid\, t\in\Fp_q\Bigr\}.
$$
\par
Let $H\subset H_q$ be a maximal subgroup. Let $D=\det(H)$ be the image
of the determinant restricted to $H$. Since
$$
H\subset \det{}^{-1}(D),
$$
we have either $\det{}^{-1}(D)=H_q$, i.e., $D=\Fp_q^{\times}$, or
$H=\det{}^{-1}(D)$. In this second case, the subgroup $D$ must be a
maximal subgroup of $\Fp_q^{\times}$ for $H$ to be maximal, which
implies that $H$ is of the form $C_{\ell}$ for some
$\ell$. Conversely, such a subgroup is maximal because if we add any
extra element $g$ and let $H'=\langle C_{\ell},g\rangle$, the fact
that $U\subset C_{\ell}$ implies that some
$$
\begin{pmatrix}
a'&0\\0&1
\end{pmatrix},
$$
with $a'\notin (\Fp_q^{\times})^{\ell}$, is in $H'$, and then by
maximality in $\Fp_q^{\times}$, we have $H'=H_q$. 
\par
Note that the $C_{\ell}$ are normal in $H_q$, hence also pairwise
non-conjugate. 
\par
In the first case, when $D=\Fp_q^{\times}$, i.e., when the determinant
restricted to $H$ is surjective, we claim that the determinant is also
injective on $H$: indeed, otherwise, there exists $u\in U\cap H$,
$u\not=1$, say
$$
u=\begin{pmatrix}1& t\\0&1\end{pmatrix}\quad\text{ with } \quad
t\not=0.
$$
\par
For any $a\in \Fp_q^*$, by surjectivity there exists
$\alpha(a)\in\Fp_q$ with
$$
\begin{pmatrix}a& \alpha(a)\\0&1\end{pmatrix}
$$
and by applying the relation
$$
\begin{pmatrix}a&\alpha(a)\\0&1
\end{pmatrix}
\cdot
\begin{pmatrix}1&t\\0&1
\end{pmatrix}
\cdot
\begin{pmatrix}a&\alpha(a)\\0&1
\end{pmatrix}^{-1}=
\begin{pmatrix}1&at\\0&1
\end{pmatrix}
$$
for all $a\in\Fp_q^{\times}$, we conclude that in fact $U\subset
H$. Then $|H|$ is divisible both by $q$ and by $q-1$, hence $H=H_q$,
contradicting the assumption that $H$ is a proper subgroup of
$H_q$. So, in this second case, the determinant gives an isomorphism
$H\simeq \Fp_q^{\times}$. Then a generator of $H$ is in one of the
conjugacy classes $g_b$ (it has distinct eigenvalues in
$\Fp_q^{\times}$), hence it is conjugate to an element (generator) of
$A$, and $H$ itself is therefore conjugate to $A$.
\end{proof}

\begin{proof}[Proof of Proposition~\ref{pr-axplusb}]
  First of all, in addition to the maximal subgroups $C_{\ell}$ given
  by Lemma~\ref{lm-max-hq}, there are subgroups $C_d$ for all
  squarefree divisors $d\mid q-1$, the inverse image under the
  determinant of the subgroup of order $(q-1)/d$ in the cyclic group
  $\Fp_q^{\times}$.
\par
Given a subset $I\subset \max(H_q)$, we now compute the density of
conjugacy classes in
$$
\Hc_I^{\sharp}=\bigcap_{\Hc\in I}{\Hc^{\sharp}},
$$
as follows:
\par
\noindent -- If $A\in I$, then with $I'=I-\{A\}$, and $d$ the product
of those primes $\ell$ for which $C_{\ell}\in I'$ (including $d=1$
when $I'=\emptyset$), we have
$$
\nu(\Hc_I^{\sharp})=\frac{1}{d}-q^{-1}\quad\text{ and in particular }
\nu(A^{\sharp})=1-q^{-1}.
$$
\par
Indeed, we have to find the density of those elements of $H_q$ which
are diagonalizable with eigenvalues $1$ and $a\in C_d$. These are
exactly the conjugacy classes $g_b^{\sharp}$ with $b\in C_d-\{1\}$,
and the trivial class, so
$$
\nu(\Hc_I^{\sharp})=\frac{1+((q-1)/d-1)q}{q(q-1)}
=\frac{q(q-1)/d-(q-1)}{q(q-1)}=\frac{1}{d}-\frac{1}{q}.
$$
\par
\noindent -- If $A\notin I$, then $I$ corresponds to a divisor $d\mid
q-1$, $d\not=1$, and we have
$$
\nu(\Hc_I^{\sharp})=\frac{1}{d},
$$
since we must now compute the density of elements of $H_q$ which have
$\det(g)\in C_d$, and this is
$$
\frac{q((q-1)/d-1)+1+q-1}{q(q-1)}=\frac{1}{d}.
$$
\par
Applying~(\ref{eq-comput}), and isolating the contribution of
$I=\{A\}$, leads exactly to~(\ref{eq-caxpb}) and to~(\ref{eq-saxpb}).
\par
To deduce~(\ref{eq-asymp-axpb}) for $\cheb{H_q}$, we may assume
$q=p^k$ with $p$ an odd prime, since for $q$ even, we have
$$
\cheb{H_q}=q+\cheb{\Zz/(q-1)\Zz}=q+O(1)
$$
by Corollary~\ref{cor-concentration}. So for $q$ odd, we write
$$
\cheb{H_q}=q+\cheb{\Zz/(q-1)\Zz}-\Delta(q)=q-\Delta(q)+O(1)
$$
where
$$
\Delta(q)=\sum_{1\not=d\mid q-1}{\frac{\mu(d)}{1-d^{-1}+q^{-1}}}.
$$
\par
Since $1-d^{-1}+q^{-1}\geq 1-d^{-1}>0$, we can bound this from above
by
$$
|\Delta(q)|\leq \sumb_{1\not=d\mid q-1}{\frac{1}{1-d^{-1}}},
$$
and then proceeding as in the proof of
Corollary~\ref{cor-concentration}, we obtain
\begin{align*}
|\Delta(q)|&\leq 
\sum_{k\geq 0}{\Bigr(\prod_{p\mid q-1}{(1+p^{-k})}-1\Bigr)}\\
&\leq \tau(q-1)+\frac{\psi(q-1)}{q-1}-2+\sum_{k\geq 2}{
\Bigl(\frac{\zeta(k)}{\zeta(2k)}-1\Bigr)
}\\
&=O(\tau(q-1))
\end{align*}
since the series converges absolutely again.
\par
Finally, the asymptotics for $\scheb{H_q}$ are obtained by essentially
identical arguments.
\end{proof}

The proof confirms the intuitive fact that the large size of
$\cheb{H_q}$ is due directly to the existence of a fairly small
diagonal subgroup $A$ (of index $q$) that contains elements conjugate
to a very large proportion of elements of $H_q$. So the waiting time
is quite close to the waiting time until a non-diagonalizable element
is obtained, which is a geometric random variable $T$ with
$$
\proba(T=k)=\frac{1}{q}\Bigl(1-\frac{1}{q}\Bigr)^{k-1},\quad
\text{ for } k\geq 1
$$
(since very often, it will be the case that sufficiently many
diagonalizable elements will have appeared by the time an element of
$U$ appears to generate the whole group).
\par
This is confirmed by the large second moment $\scheb{H_q}$: it
corresponds to a standard deviation of the waiting time which is
$$
\sqrt{\scheb{H_q}-\cheb{H_q}^2}\ \sim\ q,\quad\text{ as } q\ra +\infty,
$$
i.e., very close to the expectation, similar to the fact that
$$
\variance(T)=q\sqrt{1-\frac{1}{q}}.
$$
\par
The groups $G=H_q$ also show that the inequality~(\ref{eq-maximal}) is
best possible (with the maximal subgroup $H=A$), as observed also
in~\cite{serre-jordan}, so it is not surprising that they lead to high
Chebotarev invariants.

\section{Some finite groups of Lie type}
\label{sec-non-ab-2}

For specific complicated non-abelian groups, the Chebotarev invariant
may be hard to compute exactly, except numerically using the formulas
of Proposition~\ref{pr-comput}, when feasible (we will give examples
from computer calculations in Section~\ref{sec-ex2}).  However, if we
consider infinite families of non-abelian groups, it may be that the
subgroup structure is sufficiently well-known, simple and regular,
that one can derive asymptotic information. In fact, using results
like Proposition~\ref{pr-general-estimates}, it is not needed for this
purpose to have complete control over \emph{all} maximal subgroups.
\par
We illustrate this first simplest family of simple groups of Lie type.

\begin{theorem}\label{th-psl}
\emph{(1)}
We have
$$
\cheb{PSL(2,\Fp_p)}=3+O(p^{-1}),\quad\quad
\scheb{PSL(2,\Fp_p)}=11+O(p^{-1}),
$$
for primes $p\geq 2$. 
\par
\emph{(2)} For all $k\geq 2$, we have
$$
\proba(\tau_{PSL(2,\Fp_p)}=k)=\frac{1}{2^{k-1}}+O(p^{-1})
$$
where the implied constant depends on $k$.

\par
\emph{(3)} The same results hold for $SL(2,\Fp_p)$, and in fact
\begin{equation}\label{eq-sl-psl}
\cheb{SL(2,\Fp_p)}=\cheb{PSL(2,\Fp_p)},
\quad\quad
\scheb{SL(2,\Fp_p)}=\scheb{PSL(2,\Fp_p)} 
\end{equation}
for all $p$.
\end{theorem}

Note that the limit of $\cheb{SL(2,\Fp_p)}$ is not the minimal number
of generators of $SL(2,\Fp_p)$ (which is $2$, since $SL(2,\Fp_p)$ is
generated by the two elementary matrices with $1$ over and under the
main diagonal.)
\par
For the proof, we will not use the formula of
Proposition~\ref{pr-comput}, although this could be done at least to
prove (1) (since the subgroups of $PSL(2,\Fp_p)$ are well understood
since Dickson, see, e.g.,~\cite[Th. 2.2]{giudici}). Instead, we use
the following criterion of Serre~\cite[Prop. 19]{serre-elliptic}
(which is itself based on knowing the subgroup structure).

\begin{lemma}[Serre]\label{lm-serre}
Let $p\geq 5$ be a prime number. Assume that $G\subset SL(2,\Fp_p)$ is
a subgroup such that
\par
\emph{(1)} The group $G$ contains an element $s$ such that
$\Tr(s)^2-4$ is a non-zero square in $\Fp_p$, and such that
$\Tr(s)\not=0$;
\par
\emph{(2)} The group $G$ contains an element $s$ such that
$\Tr(s)^2-4$ is not a square in $\Fp_p$, and such that $\Tr(s)\not=0$;
\par
\emph{(3)} The group $G$ contains an element $s$ such that
  $\Tr(s)^2\in\Fp_p$  is not in $\{0,1,2,4\}$, and is not a root of
  $X^2-3X+1$.
\par
Then we have $G=SL(2,\Fp_p)$.
\end{lemma}

\begin{proof}[Proof of Theorem~\ref{th-psl}]
  We first notice that we need only check~(\ref{eq-sl-psl}) and then
  consider the case of $SL(2,\Fp_p)$. These equalities are
  consequences of Lemma~\ref{lm-frattini}, since\footnote{\ In fact,
    it is known that there is equality, but we do not need this
    stronger fact.} $\{\pm I\} \subseteq \Phi_p$, where $\Phi_p$ is
  the Frattini subgroup of $SL(2,\Fp_p)$. Indeed, we may of course
  assume that $p\neq 2$; then, if $p$ is such that $ - I \notin
  \Phi_p$, there exists a maximal subgroup $H$ of $SL(2,\Fp_p)$ which
  surjects to $PSL(2,\Fp_p)$.  We would then have
$$
SL(2,\Fp_p)= \{\pm I\} \times H
$$
which is impossible, since $SL(2,\Fp_p)$ is generated
by the elements
$$
\begin{pmatrix}1 & 1 \\0 &
    1\end{pmatrix}\quad\quad
\begin{pmatrix}1 & 0 \\1 & 1
\end{pmatrix},
$$
both of which are of odd order $p$, hence contained in $H$ (compare
with~\cite[IV-23]{serre-elladic}).
\par
Now we consider $SL(2,\Fp_p)$, and we assume $p\geq 5$. Let
$\tau=\tau_{SL(2,\Fp_p)}$ denote the corresponding waiting time, and
let $\tau_1$, $\tau_2$, $\tau_3$ denote the waiting times for
conjugacy classes satisfying the conditions (1), (2) and (3) in
Lemma~\ref{lm-serre}, i.e., for instance
$$
\tau_1=\min\{n\geq 1\,:\, s=X_n^{\sharp} \text{ has $\Tr(s)\not=0$ and
  $\Tr(s)^2-4$ is in $(\Fp_p^{\times})^2$}\}.
$$
\par
Let also $\tau^*_1$, $\tau^*_2$ be the waiting times for conditions
(1) and (2) \emph{without} the condition $\Tr(s)\not=0$.  Note that
(1) and (2) are exclusive conditions.  Moreover, each $\tau_i$ is a
geometric random variable, with parameters, respectively
\begin{equation}\label{eq-equal-coupons}
p_1=\frac{1}{2}+O(p^{-1}),\quad
p_2=\frac{1}{2}+O(p^{-1}),\quad
p_3=1+O(p^{-1}),
\end{equation}
and for $\tau_1^*$, $\tau_2^*$, the parameters are also
$$
p_1^*=\frac{1}{2}+O(p^{-1}),\quad
p_2^*=\frac{1}{2}+O(p^{-1})\ ;
$$
all these facts can be checked easily, e.g., by looking at tables of
conjugacy classes in $SL(2,\Fp_p)$ (for
instance,~\cite[p. 71]{fulton-harris}).
\par
We then have
$$
\max(\tau_1^*,\tau_2^*)\leq \tau_p\leq \max(\tau_1,\tau_2,\tau_3),
$$
where the right-hand inequality comes from Lemma~\ref{lm-serre} and
the left-hand inequality is due to the fact that the Borel subgroup
$$
B=\Bigl\{\begin{pmatrix}x&a\\0&x^{-1}
\end{pmatrix}\Bigr\}\subset SL(2,\Fp_p)
$$
intersects every conjugacy class satisfying (1) (so that $\tau_p\geq
\tau_2^*$) and the non-split Cartan subgroup
$$
C_{ns}=\Bigl\{\begin{pmatrix}
a&b\\
\eps b& a
\end{pmatrix}
\Bigr\}\subset SL(2,\Fp_p)
$$
intersects every conjugacy class satisfying (2), where
$\eps\in\Fp_p^{\times}$ is a fixed non-square element (so that
$\tau_p\geq \tau_1^*$).
\par
By applying Proposition~\ref{pr-comput} to compute the expectation and
second moment on the two extreme sides, we find
$$
3+O(p^{-1})\leq \expect(\tau_p)\leq 3+O(p^{-1}),\quad\quad
11+O(p^{-1})\leq \expect(\tau_p^2)\leq 11+O(p^{-1}).
$$
which proves (1).
\par
To prove (2), fix some $k\geq 2$. We denote
$$
\tau^*_p=\max(\tau^*_1,\tau^*_2),\quad\quad
\tau'_p=\max(\tau_1,\tau_2,\tau_3),
$$
\par
We have the equality of events
$$
\{\tau_p=k\}=\{\tau_p=\tau'_p=k\}\cup \{\tau_p=k<\tau'_p\},
$$
which is of course a disjoint union. Then we note that
$$
\proba(\tau_p=k<\tau'_p)\leq \sum_{1\leq j\leq k}
{\proba(\tau_p^*=j,\ \tau_p'>j)}.
$$
\par
But clearly, if $\tau^*_p=j$ and $\tau^*_p<\tau'_p$, either one of the
conjugacy classes $(X_1^{\sharp},\ldots, X_j^{\sharp})$ has trace
zero, or otherwise we must have $\tau'_p=\tau_3>j\geq 2$.  In the
first case, since all $X_n$ have the same uniform distribution, the
probability is at most
$$
j\proba(\Tr(X_1^{\sharp})=0)\ll jp^{-1}
$$
for all $p\geq 2$ (again by looking at conjugacy classes for
example). In the second case, we have
$$
\proba(\tau_3>j)\leq \proba(\tau_3\geq 2)\ll p^{-2}.
$$
\par
Combining this with the equality of events we found, it follows that
for $k$ fixed, we have
$$
\proba(\tau_p=k)=\proba(\tau_p=\tau'_p=k)+O(p^{-1})
$$
where the implied constant depends on $k$.
\par
Next we note that
$$
\{\tau'_p=k\}=\{\tau_p=\tau'_p=k\}\cup \{\tau'_k=p,\ \tau_p<k\},
$$
again a disjoint union. As above, we find that
$$
\proba(\tau'_k=p,\ \tau_p<k)\leq
\sum_{j=1}^{k-1}{\proba(\tau^*_p=j<\tau'_p)}
\ll p^{-1}
$$
where the implied constant depends on $k$, and hence we have finally
$$
\proba(\tau_p=k)=\proba(\tau'_p=k)+O(p^{-1}),
$$
and the result now follows easily: first, by arguments already used,
we have
$$
\proba(\tau'_p=k)=\proba(\max(\tau_1,\tau_2)=k)+O(p^{-1})
$$
and then we are left with a coupon collector problem with two coupons
of roughly equal probability by~(\ref{eq-equal-coupons}). This gives
$$
\proba(\max(\tau_1,\tau_2)=k)=
p_1^{k-1}p_2+p_2^{k-1}p_1=
2\Bigl(\frac{1}{2}+O(p^{-1})\Bigr)^{k}=\frac{1}{2^{k-1}}+O(p^{-1})
$$
for $p\geq 2$, the implied constant depending on $k$.
\end{proof}

\begin{remark}
Part (2) states that the waiting time $\tau_{SL(2,\Fp_p)}$ converges
in law, as $p\ra +\infty$, to the waiting time for a coupon collector
problem with two coupons of probability $1/2$. Intuitively, those
represent finding matrices with split or irreducible characteristic
polynomial. 
\end{remark}


\begin{remark}
  Recent results of Fulman and Guralnick (announced
  in~\cite{fulman-guralnick}) should lead to a similar good
  understanding of $\cheb{\Gg(\Fp_q)}$ when $\Gg$ is a fixed (almost
  simple) algebraic group over $\Qq$. Indeed, their results should
  also be applicable to situations with rank going to infinity, which
  are analogue of the symmetric and alternating groups that we
  consider now.\\

\end{remark}

\section{Symmetric and alternating groups}
\label{sec-sym-alt}

We now come to the case of the symmetric groups $\sy_n$ and
alternating groups $A_n$. Here we have the following result, which is
a precise formulation of a result essentially conjectured by
Dixon~\cite[Abstract]{dixon}, following McKay:\footnote{\ This
  conjecture is imprecisely formulated in~\cite{dixon}, where the
  ``expected number of elements needed to generate $\sy_n$
  invariably'' seems to mean any $r(n)$ for which
  $\proba(\cheb{\sy_n}>r(n))\ra 0$.}

\begin{theorem}\label{th-cheb-symmetric}
For $n\geq 1$, we have
$$
\cheb{\sy_n}\asymp 1,\quad\quad \cheb{A_n}\asymp 1,\quad\quad
\scheb{\sy_n}\asymp 1,\quad\quad \scheb{A_n}\asymp 1.
$$
\par
In fact, there exists a constant $c>1$ such that, for all $n\geq 1$,
we have
$$
\expect(c^{\tau_{\sy_n}})\ll 1,\quad\quad
\expect(c^{\tau_{A_n}})\ll 1.
$$
\end{theorem}

The proof is based on the following difficult result of \L uczak and
Pyber, the proof of which involves a lot of information on symmetric
groups.

\begin{theorem}[\L uczak and Pyber]\label{th-dixon}
For any $\eps>0$, there exists a constant $C$ depending only on $\eps$
such that
$$
\proba((X_1^{\sharp},\ldots,X_m^{\sharp})\text{ generate }
\sy_n)
>1-\eps
$$
for all $m\geq C$ and all $n\geq 1$. The same applies to $A_n$.
\end{theorem}

This is proved in~\cite{luczak-pyber}, improving earlier work of
Dixon~\cite{dixon}. 

\begin{proof}
  We need only prove that the exponential moments
  $\expect(c^{\tau_{n}})$ are bounded for some $c>1$, where
  $\tau_n=\tau_{G_n}$ with $G_n=\sy_n$ (the $A_n$ case is similar).
\par
From Theorem~\ref{th-dixon}, there exists $m\geq 1$ such that
\begin{equation}\label{eq-lp-bound}
\proba((Y_1^{\sharp},\ldots,Y_m^{\sharp})\text{ do not generate }
\sy_n)
\leq\frac{1}{2}
\end{equation}
for any family of independent, uniformly distributed random variables
$Y_i$ on $G_n$.
\par
Now let $k\geq 1$ be given; we can partition the set
$\{1,\ldots,k-1\}$ in $\lfloor (k-1)/m\rfloor\geq 0$ subsets of size
$m$ and a remainder, and we observe that if $\tau_n=k$, for each of
these subsets $I$, we have
$$
\proba((X_i^{\sharp}),\ i\in I)\leq \frac{1}{2},
$$
by independence and~(\ref{eq-lp-bound}). Since all those sets are
disjoints, we get 
$$
\proba(\tau_n=k)\leq \Bigl(\frac{1}{2}\Bigr)^{\lfloor (k-1)/m\rfloor}
\leq 2^{1-(k-1)/m}
$$
for $k\geq 1$, and then, for any $c\geq 1$, we have
$$
\expect(c^{\tau_n})=\sum_{k\geq 1}{c^k\proba(\tau_n=k)}
\leq 2^{1+1/m}\sum_{k\geq 1}{(c2^{1/m})^k}
$$
which converges, and is independent of $n$, for any $c$ with
$1<c<2^{1/m}$. 
\end{proof}

In view of this, the following question seems natural:

\begin{question}
Is it true or not that for \emph{all} $c>1$, we have
$$
\expect(c^{\tau_{\sy_n}})\ll 1
$$
for $n\geq 1$ (and similarly for $A_n$)?
\end{question}


Another natural question, also suggested by Dixon, is:

\begin{question}\label{q-limit}
  Do the sequences $(\cheb{\sy_n})$ (or $(\cheb{A_n})$)
  converge as $n\ra +\infty$? If they do, can their limits be
  computed?
\end{question}

Our guess is that the answer is positive. In fact, we now present a
heuristic model that suggests this and predicts the value of
$$
\lim_{n\ra +\infty}{\cheb{A_n}}.
$$
\par
Our first step is to apply Corollary~\ref{cor-asymptotic} to a
suitable ``essential'' set of maximal subgroups of symmetric groups to
obtain a simpler waiting time that is asymptotically close to
$\cheb{A_n}$ (or to $\cheb{\sy_n}$). The required result is again one
due to \L uczak and Pyber~\cite{luczak-pyber}.

\begin{theorem}[\L uczak and Pyber]\label{th-lp-2}
For $n\geq 1$, let $S_n$ be the set of $g\in\sy_n$ such that
$g$ is contained in a subgroup $H$ of $\sy_n$, distinct from $A_n$,
and such that $G$ acts transitively on $\{1,\ldots, n\}$. Then we have
$$
\lim_{n\ra +\infty}{\nu_n(S_n)}=0,
$$
where $\nu_n(A)=|A|/|\sy_n|$ is the uniform density on the
symmetric group.
\end{theorem}

\begin{corollary}\label{cor-partial-alt}
For $n\geq 1$ and $1\leq i< n/2$, let
$$
H_{i,n}=\big\{g\in\sy_n\,\mid\, g\cdot \{1,\ldots,
i\}=\{1,\ldots, i\} \big\}
$$
be the subgroup of $\sy_n$ leaving $\{1,\ldots, i\}$
invariant. Let $H'_{i,n}=H_{i,n}\cap A_n$. Then the $H_{i,n}$ --
resp. $H'_{i,n}$ -- are maximal subgroups of $\sy_n$ --
resp. $A_n$ --. Moreover, let
\begin{gather*}
M_n=\{A_n\}\cup \{H_{i,n}\,\mid\, 1\leq i < n/2\}\subset
\max(\sy_n), \\
M'_n=\{H'_{i,n}\,\mid\, 1\leq i< n/2\}\subset \max(A_n).
\end{gather*}
\par
As in Proposition~\ref{pr-general-estimates}, let $\tilde{\tau}_n$,
resp. $\tilde{\tau}'_n$, be the waiting time before conjugacy classes
in each subgroup of $M_n$, resp. $M'_n$, has been observed. Then we
have
$$
\cheb{\sy_n}=\expect(\tilde{\tau}_n)+o(1),\quad\quad
\scheb{\sy_n}=\expect(\tilde{\tau}_n^2)+o(1),
$$
as $n\ra +\infty$, and similarly
$$
\cheb{A_n}=\expect(\tilde{\tau}'_n)+o(1),\quad\quad
\scheb{A_n}=\expect((\tilde{\tau}'_n)^2)+o(1).
$$
\end{corollary}

\begin{proof}
  It is known that the $H_{i,n}$ are (representatives of) the
  conjugacy classes of maximal intransitive subgroups of
  $\sy_n$. Thus, we find by definition of $S_n$ that
$$
\bigcup_{\Hc\in\max(\sy_n)-M_n}{\Hc^{\sharp}}=S_n,
$$
and hence the result follows immediately from
Corollary~\ref{cor-asymptotic} and Theorem~\ref{th-lp-2}, which
provides us with the assumption~(\ref{eq-cond-asymp}) required.
\end{proof}

In particular, in approaching Question~\ref{q-limit}, it is enough to
consider the expectations and second moment of the random variables
$\tilde{\tau}_n$ and $\tilde{\tau}'_n$. Those are combinatorially
simpler, or at least more explicit.
\par
In particular, note the following: an element $\sigma\in
\mathfrak{S}_n$ is conjugate to an element of $H_{i,n}\subset
\mathfrak{S}_n$ \emph{if and only if}, when expressed as a product of
disjoint cycles of lengths $\ell_j(\sigma)\geq 1$, $1\leq j\leq
\varpi(\sigma)$, say, has the property that a sum of a subset of the
lengths is equal to $i$: for some $J\subset \{1,\ldots,
\varpi(\sigma)\}$, we have
$$
\sum_{j\in J}{\ell_j(\sigma)}=i.
$$
\par
Note also that this applies equally to an element $\sigma$ in $A_n$:
the element is conjugate to $H'_{i,n}\subset A_n$ if and only if the
property above is true for its cycle lengths computed in
$\mathfrak{S}_n$ (although these cycle lengths do not always
characterize the conjugacy class of $\sigma$ in $A_n$).
\par
In particular, conjugacy classes
$(\sigma_1^{\sharp},\ldots,\sigma_k^{\sharp})$ in $\sy_n^{\sharp}$ (or
$A_n^{\sharp}$) generate a transitive subgroup of $\sy_n$ (or $A_n$)
\emph{if and only} if $n$ (which always occurs as the sum of all
lengths) is the only common such sum for all $\sigma_j$. (Indeed, if
$i<n$ occurs as a common subsum, we can assume that $i\leq n/2$, and
then it is possible to select elements in each conjugacy class which
all belong to $H_{i,n}$, so that the conjugacy classes can not
generate invariably a transitive subgroup; the converse is also
simple.)
\par
Now, we come to the model when $n\ra +\infty$. The distribution of the
set of lengths of random permutations is a well-studied subject in
probabilistic group theory, and this allows us to make a guess as to
the existence and value of the limit.
\par
Indeed, for $i\geq 1$, consider the map
$$
\varpi_i\,:\, \sy_n\ra \{0,1,\ldots\}
$$
sending $\sigma$ to the number of cycles of length $i$ in its
decomposition as product of disjoint cycles (for $i=1$, this is the
number of fixed points; for $i\geq n+1$, of course, this is zero, but
it will be convenient for the asymptotic study to allow arbitrary
$i$).  Now consider, for each $n\geq 1$, any random variables $s_n$,
$\sigma_n$ uniformly distributed on $\sy_n$ and $A_n$,
respectively. Then the following is a consequence of well-known
results dating back to Goncharov~\cite{goncharov}:
for fixed $i$, as $n\ra +\infty$, the random variables
$\varpi_i(\sigma_n)$ converge in law to a Poisson random variable with
parameter $1/i$, i.e., we have
\begin{equation}\label{eq-goncharov-alternating}
\lim_{n\ra +\infty}{\proba(\varpi_i(\sigma_n)=k)}=
e^{-1/i}\frac{1}{k!i^k},\quad\quad \text{ for fixed $k\geq 0$.}
\end{equation}
\par
Moreover, the limits for distinct values of $i$ are independent, i.e.,
for any fixed finite set $I$ of positive integers, we have
$$
\lim_{n\ra +\infty}{\proba(\varpi_i(\sigma_n)=k_i\text{ for all } i\in
  I)}=
\prod_{i\in I}{e^{-1/i}\frac{1}{i^{k_i}k_i!}}.
$$
\par
More precisely, this is proved (and much more precise results) for
symmetric groups in, e.g.,~\cite[Th.1 ]{arratia-tavare}
or~\cite[Th. 1.3]{abt}. The case of alternating groups can be deduced
from this using the fact that the indicator function of $A_n$ in
$\sy_n$ is given in terms of cycle-lengths by
$$
\frac{1+(-1)^{\varpi_2+\varpi_4+\cdots}}{2}.
$$
\par
For instance, for fixed $j$, the characteristic function of
$\varpi_{k}(\sigma_n)$ is 
$$
\expect(e^{it\varpi_k(\sigma_n)})=
\expect(e^{it\varpi_k(s_n)})+
\expect((-1)^{\sum_{j}{\varpi_{2j}(s_n)}}e^{it\varpi_k(s_n)}).
$$
\par
By Goncharov's result, the first term converges for every $t\in\Rr$ to
the desired characteristic function of a Poisson variable with
parameter $1/k$; for the second term, we can use the method of Lloyd
and Shepp~\cite[\S 2]{lloyd-shepp}. Assuming $k=2k'$ is even (the
other case being similar), one finds (see in
particular~\cite[(3)]{lloyd-shepp}) that the expectation over $\sy_n$
is the coefficient of $z^{n}$ in
$$
\frac{1}{1-z}
\exp\Bigl(\frac{z^{k}}{k}(e^{i(t+\pi)}-1)\Bigr)
\prod_{\stacksum{j\geq 1}{j\not=k'}}{\exp(-z^{2j}/j)}
=(1+z)\exp\Bigl(\frac{z^{k}}{k}(1-e^{it})\Bigr),
$$
and since this function (of $z\in\Cc$) is regular at $z=1$, those
coefficients converge to $0$ for every fixed $t$. This computation
proves~(\ref{eq-goncharov-alternating}).
\par
It seems therefore reasonable to use a model of Poisson variables to
predict the limit of Chebotarev invariants of alternating groups.  For
this purpose, let $\mathcal{A}$ be the set of sequences
$(\ell_i)_{i\geq 1}$ of non-negative integers; we denote the $i$-th
component of $\ell\in \mathcal{A}$ by $\varpi_i(\ell)$. Let
$\nu_{\mathcal{A}}$ be the infinite product (probability) measure on
$\mathcal{A}$ such that the $i$-th component $\ell_i$ is distributed
like a Poisson random variable with parameter $1/i$. This set
$\mathcal{A}$ is meant to be like the set of conjugacy classes of an
infinite symmetric group, and indeed, from the above, we see that for
any finite $I$ of positive integers and any $k_i\geq 0$ defined for
$i\in I$, we have
$$
\lim_{n\ra +\infty}{\proba(\varpi_i(\sigma_n)=k_i\text{ for all } i\in
  I)}=\nu_{\mathcal{A}}(\{\ell\in\mathcal{A}\,\mid\, \varpi_i(\ell)=k_i,\ i\in I\}).
$$
\par
Now consider an infinite sequence $(X_k)_{k\geq 1}$ of $\mathcal{A}$-valued,
independent random variables, identically distributed according to
$\nu$. We look at the following waiting time: 
$$
\tau_{\mathcal{A}}=\min\{k\geq 1\,\mid\, \bigcap_{1\leq j\leq
  k}{S(X_j)} =\{+\infty\}\},
$$
where, for $\ell\in\mathcal{A}$, we denote by $S(\ell)\subset \{0,1,2,\ldots,
\}\cup\{+\infty\}$ the set of all sums
$$
\sum_{i\geq 1}{ib_i},\quad\quad \text{ where } 0\leq b_i\leq
\varpi_i(\ell)
$$
(note the usual shift of notation from our description of the case of
fixed $n$: the sequence of lengths of cycles occuring in a permutation
is replaced by the sequence of multiplicities of each possible
length). Then our guess for the limit of $\cheb{A_n}$ is that
$$
\lim_{n\ra +\infty}{\cheb{A_n}}=\expect(\tau_A).
$$
\par
We hope to come back to this  question in a future work.

\section{Non-abelian groups: numerical experiments}\label{sec-ex2}

In this section, we give some tables of values of the Chebotarev
invariant (and the secondary invariant) for some non-abelian finite
groups.  Although those are clearly rational numbers, we list real
approximations only because the ``height'' of those rationals grows
very fast.
\par
The computations are feasible even for fairly large and complicated
non-abelian groups, because they may have few conjugacy classes of
maximal subgroups, and not too many conjugacy classes. For instance,
the Weyl group $W(E_8)$ (one of our motivating examples) has $9$
conjugacy classes of maximal subgroups, and $112$ conjugacy
classes. If these data are available to suitable software packages,
Proposition~\ref{pr-comput} provides a way to compute the Chebotarev
invariants, though this is at best an exponential-time algorithm (due
to the necessity to sum over all subsets of $\max(G)$).
\par
The computations here were done for the most part with \textsc{Magma}
(see~\cite{magma}), using the script included in the Appendix. The
correctness of the results was checked partly by independent
computations with the open-source package \textsc{GAP}
(see~\cite{gap}), and by checking that the results agree, for cyclic
groups and groups $\Fp_p^k$, with the theoretical formulas of the
Section~\ref{sec-ex1}. They are also in good agreement, in the case of
$PSL(2,\Fp_p)$, with the asymptotic result of
Section~\ref{sec-non-ab-2}. Hence, altogether, we have very high
confidence in these values.
\par
The computations were relatively fast; usually there was a sharp
threshold between computing for one group in a family in less than an
hour, and the next one proving infeasible due to the exponential
growth of the number of subsets of $\max(G)$. As an indication of
timing, the computation for $PSL(6,\Fp_3)$ with \textsc{Magma}
(version 2.14.15) took about $42$ seconds on a $2.5$ GHz Core 2
processor.
\par
Below, we include tables for the alternating groups $A_n$, for the
symmetric groups $\sy_n$, for the groups $PSL(2,\Fp_p)$ with $p$ prime
$\leq 150$ (though the computations can be done for $p$ quite a bit
larger, we do not include the results which are not particularly
enlightening), for $PSL(3,\Fp_p)$, $PSL(4,\Fp_p)$, $PSL(n,\Fp_2)$,
$PSL(n,\Fp_3)$, $PSL(n,\Fp_4)$, $PSL(2,\Fp_{2^n})$,
$PSL(3,\Fp_{2^n})$, $Sp(2g,\Fp_3)$. (Note that, in general, the
computations tend to run quite a bit faster for simple groups.) We
also include a table of the ``partial'' invariants
$\expect(\tilde{\tau}'_n)$ and $\expect((\tilde{\tau}'_n)^2)$ of
alternating groups defined in Corollary~\ref{cor-partial-alt}. Note
that although we have shown that these are asymptotically converging
to the Chebotarev invariants themselves, the convergence is by no
means visible! There is also a table for the Borel subgroup of
$SL(3,\Fp_p)$, namely
$$
B_3(\Fp_p)=\Bigl\{
\begin{pmatrix}
x & r & s\\
0 & y & t\\
0 & 0 & z
\end{pmatrix}\,\mid\, 
(r,s,t)\in \Fp_p^3,\ 
(x,y,z)\in (\Fp_p^{\times})^3,\ xyz=1
\Bigr\},
$$
as another example of a solvable group. 
\par
Another table lists some more ``sporadic'' groups; the names of those
groups in the table should be self-explanatory. For instance, $D_{2n}$
is the dihedral group of order $2n$, $W(R)$ denotes the Weyl group of
a root system of type $R$; $Sz(8)$ and $Sz(32)$ are Suzuki
groups. Sporadic simple groups are named in a standard way:
\begin{itemize}
\item Mathieu groups: $M_{n}$, where $n\in \{11,12,22,23,24\}$;
\item Janko groups: $J_k$, where $k\in \{1,2,3,4\}$;
\item Second and thirs Conway groups: $Co_2$, $Co_3$ (the first Conway
  group is too big);
\item Tits group $T$;
\item MacLaughlin group $McL$;
\item Higman-Sims group $HS$;
\item Helde group $He$.
\end{itemize}
\par
The group $Rub$ at the end of the table is the Rubik's group (the
subgroup of $\sy_{48}$ that gives the possible moves on the Rubik's
Cube; computing $\cheb{Rub}$ takes about two days on a fast Opteron;
this group has $20$ conjugacy classes of maximal subgroups and $81120$
conjugacy classes). In order to ease checking, the url
\begin{center}
\url{http://www.math.ethz.ch/~kowalski/other-groups.mgm}
\end{center}
contains a \textsc{Magma} file where each group in this list is
constructed explicitly.
\par
It also possible to exploit the databases of small groups, or of
transitive groups, or primitive groups, to compute the Chebotarev
invariants for, say, all groups of a given small order (up to
isomorphism), or for all transitive permutations groups of small
degree. The latter is of course particularly interesting from the
point of view of Galois theory, and the groups $\Fp_q\rtimes
\Fp_q^{\times}$ which appear as transitive permutation groups of
degree $q$ (and in Galois theory as Galois groups of Kummer extensions
of prime-power degree, i.e., splitting fields of polynomials of the
type $X^q-a$) are very noticeable, having much higher Chebotarev
invariants than the other groups despite their rather small order (see
the example in the table for transitive groups of degree $17$ --
noting that the group with Chebotarev invariant roughly $8.88$ is the
index $2$ subgroup of $H_{17}$ denoted $C_2$ in
Section~\ref{sec-solvable}). We include a figure of the empirical
distribution of values for the Chebotarev waiting time for $H_{31}$
(chosen because $q-1=30$ has ``many'' divisors).
\par
We also include a figure with an histogram showing the distribution of
the Chebotarev invariant for the $840$ distinct groups of order $720$
(up to isomorphism). Note that this data also indicates that the
invariant is far from injective (as can be guessed from its dependency
on relatively little data): there are only $188$ distinct values of
$\cheb{G}$ for $|G|=720$; 
the value
$$
\frac{469589438194474533813031879}{80462083849550829871525080}\simeq
5.836158\ldots.
$$
occurs with maximal multiplicity (it arises $39$ times).
\par
Note that for simple groups (or groups which are nearly so), the
relation between $\cheb{G}$ and $\scheb{G}$ seems relatively regular,
but there is certainly no strict monotony in terms of the order; see,
e.g., the cases of alternating groups $A_n$, where sorting according
to the value of $\cheb{A_n}$ leads to the
following rather bizarre ordering of the segment $2\leq n\leq 21$:
$$
2, 3, 13, 19, 17, 11, 5, 10, 14, 20, 21, 16, 18, 15, 4, 6, 12, 9, 7, 8\ ;
$$
the ordering with respect to $\scheb{A_n}$ is slightly different,
namely:
$$
2, 3, 13, 19, 11, 17, 10, 14, 21, 20, 16,18, 15, 5,12,6,9,4,7,8.
$$
\par
And the orderings for $\cheb{\sy_n}$ and
$\scheb{\sy_n}$ are also different:
$$
2,3,7,11,13,9,17,19,5,15,21,16,20,4,14,18,12,10,8,6,
$$
and
$$
2,3,7,11,13,9,17,19,15,5,21,16,20,14,18,12,8,10,4,6,
$$
respectively. Note however that in Table~2, if we fix the parity of
$n$, the invariants $\expect(\tilde{\tau}'_{2n})$ and
$\expect(\tilde{\tau}'_{2n+1})$ seem monotonically increasing. This
indicates that they are indeed very natural objects to study.


\par
\begin{table}[p]
\centering
\caption{Chebotarev invariants of $A_n$}
\begin{tabular}{c|c|c|c}
$n$ & Order & $\cheb{A_n}$ & $\scheb{A_n}$ \\
\hline
$2$ & 1 & 1.000000\ldots & 1.000000\ldots\\ 
$3$ & 3 & 1.500000\ldots & 3.000000\ldots\\ 
$4$ & 12 & 4.409091\ldots & 29.71074\ldots\\ 
$5$ & 60 & 4.136364\ldots & 22.64463\ldots\\ 
$6$ & 360 & 4.439574\ldots & 25.49003\ldots\\ 
$7$ & 2520 & 4.782001\ldots & 29.98671\ldots\\ 
$8$ & 20160 & 4.939097\ldots & 31.98434\ldots\\ 
$9$ & 181440 & 4.637463\ldots & 26.35009\ldots\\ 
$10$ & 1814400 & 4.145282\ldots & 21.73709\ldots\\ 
$11$ & 19958400 & 4.092974\ldots & 21.08692\ldots\\ 
$12$ & 239500800 & 4.444074\ldots & 24.14188\ldots\\ 
$13$ & 3113510400 & 4.016324\ldots & 20.51475\ldots\\ 
$14$ & 43589145600 & 4.212753\ldots & 22.16514\ldots\\ 
$15$ & 653837184000 & 4.289698\ldots & 22.51291\ldots\\ 
$16$ & 10461394944000 & 4.239141\ldots & 22.21416\ldots\\ 
$17$ & 177843714048000 & 4.089704\ldots & 21.12890\ldots\\ 
$18$ & 3201186852864000 & 4.248133\ldots & 22.38035\ldots\\ 
$19$ & 60822550204416000 & 4.072274\ldots & 21.08656\ldots\\ 
$20$ & 1216451004088320000 & 4.229094\ldots & 22.20516\ldots\\ 
$21$ & 25545471085854720000 & 4.238026\ldots & 22.19523\ldots\\ 
$22$ & 562000363888803840000 & 4.240513\ldots & 22.33370\ldots\\ 
$23$ & 12926008369442488320000 & 4.131077\ldots & 21.54514\ldots\\ 
$24$ & 310224200866619719680000 & 4.282667\ldots & 22.58460\ldots\\ 
\end{tabular}
\end{table}
\par

\begin{table}[p]
\centering
\caption{``Partial'' Chebotarev invariants of $A_n$}
\begin{tabular}{c|c|c|c}
$n$ & Order & $\expect(\tilde{\tau}'_n)$ & $\expect((\tilde{\tau}'_n)^2)$ \\
\hline
$3$ & 3 & 1.500000\ldots & 3.000000\ldots\\ 
$4$ & 12 & 2.123377\ldots & 5.874009\ldots\\ 
$5$ & 60 & 2.500000\ldots & 10.00000\ldots\\ 
$6$ & 360 & 2.649424\ldots & 9.187574\ldots\\ 
$7$ & 2520 & 3.243247\ldots & 16.47701\ldots\\ 
$8$ & 20160 & 2.812743\ldots & 10.71084\ldots\\ 
$9$ & 181440 & 3.133704\ldots & 13.97383\ldots\\ 
$10$ & 1814400 & 3.115450\ldots & 13.08967\ldots\\ 
$11$ & 19958400 & 3.399573\ldots & 15.88920\ldots\\ 
$12$ & 239500800 & 3.225496\ldots & 14.16483\ldots\\ 
$13$ & 3113510400 & 3.402011\ldots & 15.56383\ldots\\ 
$14$ & 43589145600 & 3.357361\ldots & 15.13742\ldots\\ 
$15$ & 653837184000 & 3.504050\ldots & 16.37350\ldots\\ 
$16$ & 10461394944000 & 3.385358\ldots & 15.32752\ldots\\ 
$17$ & 177843714048000 & 3.544719\ldots & 16.55867\ldots\\ 
$18$ & 3201186852864000 & 3.497980\ldots & 16.21775\ldots\\ 
$19$ & 60822550204416000 & 3.625919\ldots & 17.22183\ldots\\ 
$20$ & 1216451004088320000 & 3.530703\ldots & 16.46076\ldots\\ 
\end{tabular}
\end{table}

\begin{table}[p]
\centering
\caption{Chebotarev invariants of $\sy_n$}
\begin{tabular}{c|c|c|c}
$n$ & Order & $\cheb{\sy_n}$ & $\scheb{\sy_n}$ \\
\hline
$2$ & 2 & 2.000000\ldots & 6.000000\ldots\\ 
$3$ & 6 & 3.800000\ldots & 19.32000\ldots\\ 
$4$ & 24 & 4.498380\ldots & 25.91538\ldots\\ 
$5$ & 120 & 4.331526\ldots & 23.50351\ldots\\ 
$6$ & 720 & 5.610738\ldots & 37.63260\ldots\\ 
$7$ & 5040 & 4.115230\ldots & 21.20184\ldots\\ 
$8$ & 40320 & 4.626289\ldots & 25.71722\ldots\\ 
$9$ & 362880 & 4.250355\ldots & 22.49197\ldots\\ 
$10$ & 3628800 & 4.624666\ldots & 25.76898\ldots\\ 
$11$ & 39916800 & 4.173683\ldots & 21.86294\ldots\\ 
$12$ & 479001600 & 4.583705\ldots & 25.11338\ldots\\ 
$13$ & 6227020800 & 4.213748\ldots & 22.21319\ldots\\ 
$14$ & 87178291200 & 4.508042\ldots & 24.57963\ldots\\ 
$15$ & 1307674368000 & 4.365718\ldots & 23.39257\ldots\\ 
$16$ & 20922789888000 & 4.461633\ldots & 24.12713\ldots\\ 
$17$ & 355687428096000 & 4.282141\ldots & 22.79488\ldots\\ 
$18$ & 6402373705728000 & 4.531784\ldots & 24.67680\ldots\\ 
$19$ & 121645100408832000 & 4.308469\ldots & 23.01145\ldots\\ 
$20$ & 2432902008176640000 & 4.497047\ldots & 24.37207\ldots\\ 
$21$ & 51090942171709440000 & 4.391209\ldots & 23.61488\ldots\\ 
$22$ & 1124000727777607680000 & 4.477492\ldots & 24.29632\ldots\\ 
$23$ & 25852016738884976640000 & 4.352364\ldots & 23.37533\ldots\\ 
$24$ & 620448401733239439360000 & 4.523388\ldots & 24.57409\ldots\\ 
\end{tabular}
\end{table}

\begin{table}[p]
\centering
\caption{Chebotarev invariants of transitive groups of degree $17$}
\begin{tabular}{c|c|c|c}
  Name & Order & $\cheb{G}$ & $\scheb{G}$ \\
  \hline
  $\Zz/17\Zz$ & 17 & 1.062500\ldots & 1.195312\ldots\\ 
  $C_8\subset H_{17}$ & 34 & 3.094697\ldots & 11.81350\ldots\\ 
  $C_4\subset H_{17}$ & 68 & 4.890000\ldots & 35.53580\ldots\\ 
  $C_2\subset H_{17}$  & 136 & 8.880953\ldots & 138.3764\ldots\\ 
  $H_{17}$ & 272 & 17.21053\ldots & 562.3851\ldots\\ 
  $PSL(2,\Fp_{16})$ & 4080 & 3.200912\ldots & 12.73727\ldots\\ 
  $7$ & 8160 & 4.055261\ldots & 20.84364\ldots\\ 
  $8$ & 16320 & 4.067118\ldots & 20.58582\ldots\\ 
  $A_{17}$ & 177843714048000 & 4.089704\ldots & 21.12890\ldots\\ 
  $\sy_{17}$ & 355687428096000 & 4.282141\ldots & 22.79488\ldots\\ 
\end{tabular}
\end{table}

\begin{table}[p]
\centering
\caption{Chebotarev invariants of $PSL(3,\Fp_p)$}
\begin{tabular}{c|c|c|c}
$p$ & Order & $\cheb{PSL(3,\Fp_p)}$ & $\scheb{PSL(3,\Fp_p)}$ \\
\hline
$2$ & 168 & 4.653153\ldots & 29.48762\ldots\\ 
$3$ & 5616 & 3.845890\ldots & 20.67132\ldots\\ 
$5$ & 372000 & 3.629464\ldots & 18.36114\ldots\\ 
$7$ & 1876896 & 3.661481\ldots & 18.91957\ldots\\ 
$11$ & 212427600 & 3.527819\ldots & 17.29354\ldots\\ 
$13$ & 270178272 & 3.546344\ldots & 17.55063\ldots\\ 
$17$ & 6950204928 & 3.511708\ldots & 17.12456\ldots\\ 
$19$ & 5644682640 & 3.521753\ldots & 17.25893\ldots\\ 
$23$ & 78156525216 & 3.506462\ldots & 17.06878\ldots\\ 
$29$ & 499631102880 & 3.504076\ldots & 17.04348\ldots\\ 
$31$ & 283991644800 & 3.508213\ldots & 17.09800\ldots\\ 
$37$ & 1169948144736 & 3.505795\ldots & 17.06906\ldots\\ 
$41$ & 7980059337600 & 3.502051\ldots & 17.02191\ldots\\ 
\end{tabular}
\end{table}

\begin{table}[p]
\centering
\caption{Chebotarev invariants of $PSL(4,\Fp_p)$}
\begin{tabular}{c|c|c|c}
$p$ & Order & $\cheb{PSL(4,\Fp_p)}$ & $\scheb{PSL(4,\Fp_p)}$ \\
\hline
$2$ & 20160 & 4.939097\ldots & 31.98434\ldots\\ 
$3$ & 6065280 & 4.191257\ldots & 23.35082\ldots\\ 
$5$ & 7254000000 & 3.768197\ldots & 18.89633\ldots\\ 
$7$ & 2317591180800 & 3.613602\ldots & 17.31973\ldots\\ 
$11$ & 2069665112592000 & 3.530797\ldots & 16.44109\ldots\\ 
$13$ & 12714519233969280 & 3.513963\ldots & 16.24990\ldots\\ 
\end{tabular}
\end{table}

\begin{table}[p]
\centering
\caption{Chebotarev invariants of $PSL(n,\Fp_2)$}
\begin{tabular}{c|c|c|c}
$n$ & Order & $\cheb{PSL(n,\Fp_2)}$ & $\scheb{PSL(n,\Fp_2)}$ \\
\hline
$2$ & 6 & 3.800000\ldots & 19.32000\ldots\\ 
$3$ & 168 & 4.653153\ldots & 29.48762\ldots\\ 
$4$ & 20160 & 4.939097\ldots & 31.98434\ldots\\ 
$5$ & 9999360 & 4.238182\ldots & 25.64374\ldots\\ 
$6$ & 20158709760 & 4.456089\ldots & 27.20052\ldots\\ 
$7$ & 163849992929280 & 4.335957\ldots & 26.54874\ldots\\ 
$8$ & 5348063769211699200 & 4.465723\ldots & 27.53266\ldots\\ 
$9$ & 699612310033197642547200 & 4.460433\ldots & 27.64706\ldots\\ 
\end{tabular}
\end{table}

\begin{table}[p]
\centering
\caption{Chebotarev invariants of $PSL(n,\Fp_3)$}
\begin{tabular}{c|c|c|c}
$n$ & Order & $\cheb{PSL(n,\Fp_3)}$ & $\scheb{PSL(n,\Fp_3)}$ \\
\hline
$2$ & 12 & 4.409091\ldots & 29.71074\ldots\\ 
$3$ & 5616 & 3.845890\ldots & 20.67132\ldots\\ 
$4$ & 6065280 & 4.191257\ldots & 23.35082\ldots\\ 
$5$ & 237783237120 & 3.949889\ldots & 21.81110\ldots\\ 
$6$ & 21032402889738240 & 4.123378\ldots & 23.06449\ldots\\ 
$7$ & 67034222101339041669120 & 4.066340\ldots & 22.81370\ldots\\ 
\end{tabular}
\end{table}

\begin{table}[p]
\centering
\caption{Chebotarev invariants of $PSL(n,\Fp_4)$}
\begin{tabular}{c|c|c|c}
$n$ & Order & $\cheb{PSL(n,\Fp_4)}$ & $\scheb{PSL(n,\Fp_4)}$ \\
\hline
$2$ & 60 & 4.136364\ldots & 22.64463\ldots\\ 
$3$ & 20160 & 4.399979\ldots & 26.39681\ldots\\ 
$4$ & 987033600 & 3.770618\ldots & 19.19928\ldots\\ 
$5$ & 258492255436800 & 3.838194\ldots & 20.33428\ldots\\ 
$6$ & 361310134959341568000 & 4.002927\ldots & 21.57223\ldots\\ 
\end{tabular}
\end{table}

\begin{table}[p]
\centering
\caption{Chebotarev invariants of $PSL(2,\Fp_{2^n})$}
\begin{tabular}{c|c|c|c}
$2^n$ & Order & $\cheb{PSL(2,\Fp_{2^n})}$ & $\scheb{PSL(2,\Fp_{2^n})}$ \\
\hline
$2$ & 6 & 3.800000\ldots & 19.32000\ldots\\ 
$4$ & 60 & 4.136364\ldots & 22.64463\ldots\\ 
$8$ & 504 & 3.437879\ldots & 14.95188\ldots\\ 
$16$ & 4080 & 3.200912\ldots & 12.73727\ldots\\ 
$32$ & 32736 & 3.096876\ldots & 11.82191\ldots\\ 
$64$ & 262080 & 3.048732\ldots & 11.40623\ldots\\ 
$128$ & 2097024 & 3.023623\ldots & 11.19773\ldots\\ 
$256$ & 16776960 & 3.011765\ldots & 11.09826\ldots\\ 
$512$ & 134217216 & 3.005965\ldots & 11.04945\ldots\\ 
\end{tabular}
\end{table}

\begin{table}[p]
\centering
\caption{Chebotarev invariants of $PSL(3,\Fp_{2^n})$}
\begin{tabular}{c|c|c|c}
$2^n$ & Order & $\cheb{PSL(3,\Fp_{2^n})}$ & $\scheb{PSL(3,\Fp_{2^n})}$ \\
\hline
$2$ & 168 & 4.653153\ldots & 29.48762\ldots\\ 
$4$ & 20160 & 4.399979\ldots & 26.39681\ldots\\ 
$8$ & 16482816 & 3.551417\ldots & 17.54363\ldots\\ 
$16$ & 1425715200 & 3.549690\ldots & 17.47208\ldots\\ 
$32$ & 1098404364288 & 3.503357\ldots & 17.03581\ldots\\ 
\end{tabular}
\end{table}

\begin{table}[p]
\centering
\caption{Chebotarev invariants of the Borel subgroup of $SL(3,\Fp_p)$}
\begin{tabular}{c|c|c|c}
$p$ & Order & $\cheb{B_3(\Fp_{p})}$ & $\scheb{B_3(\Fp_{p})}$ \\
\hline
$2$ & 8 & 3.333333\ldots & 13.55556\ldots\\ 
$3$ & 108 & 5.074442\ldots & 31.76009\ldots\\ 
$5$ & 2000 & 7.686557\ldots & 81.14365\ldots\\ 
$7$ & 12348 & 10.07528\ldots & 150.8724\ldots\\ 
$11$ & 133100 & 16.38777\ldots & 402.7223\ldots\\ 
$13$ & 316368 & 18.85106\ldots & 551.0363\ldots\\ 
$17$ & 1257728 & 25.31072\ldots & 978.0196\ldots\\ 
$19$ & 2222316 & 27.79352\ldots & 1204.483\ldots\\ 
$23$ & 5888828 & 34.28491\ldots & 1805.763\ldots\\ 
$29$ & 19120976 & 43.27249\ldots & 2885.634\ldots\\ 
$31$ & 26811900 & 45.75644\ldots & 3268.081\ldots\\ 
$37$ & 65646288 & 54.75057\ldots & 4678.007\ldots\\ 
$41$ & 110273600 & 61.26132\ldots & 5801.515\ldots\\ 
$43$ & 140250348 & 63.74680\ldots & 6339.956\ldots\\ 
\end{tabular}
\end{table}

\begin{table}[p]
\centering
\caption{Chebotarev invariants of $PSL(2,\Fp_p)$, $p\leq 150$}
\begin{tabular}{c|c|c|c}
$p$ & Order & $\cheb{PSL(2,\Fp_p)}$ & $\scheb{PSL(2,\Fp_p)}$ \\
\hline
$2$ & 6 & 3.800000\ldots & 19.32000\ldots\\ 
$3$ & 12 & 4.409091\ldots & 29.71074\ldots\\ 
$5$ & 60 & 4.136364\ldots & 22.64463\ldots\\ 
$7$ & 168 & 4.653153\ldots & 29.48762\ldots\\ 
$11$ & 660 & 3.981397\ldots & 20.76193\ldots\\ 
$13$ & 1092 & 3.293965\ldots & 13.63659\ldots\\ 
$17$ & 2448 & 3.264353\ldots & 13.20732\ldots\\ 
$19$ & 3420 & 3.259202\ldots & 13.08533\ldots\\ 
$23$ & 6072 & 3.136600\ldots & 12.18536\ldots\\ 
$29$ & 12180 & 3.115633\ldots & 11.99619\ldots\\ 
$31$ & 14880 & 3.111661\ldots & 11.92578\ldots\\ 
$37$ & 25308 & 3.088522\ldots & 11.75723\ldots\\ 
$41$ & 34440 & 3.098342\ldots & 11.78358\ldots\\ 
$43$ & 39732 & 3.071689\ldots & 11.61064\ldots\\ 
$47$ & 51888 & 3.065454\ldots & 11.55651\ldots\\ 
$53$ & 74412 & 3.060208\ldots & 11.51103\ldots\\ 
$59$ & 102660 & 3.051900\ldots & 11.43952\ldots\\ 
$61$ & 113460 & 3.051897\ldots & 11.43943\ldots\\ 
$67$ & 150348 & 3.045600\ldots & 11.38545\ldots\\ 
$71$ & 178920 & 3.046777\ldots & 11.38343\ldots\\ 
$73$ & 194472 & 3.042989\ldots & 11.36306\ldots\\ 
$79$ & 246480 & 3.043013\ldots & 11.34889\ldots\\ 
$83$ & 285852 & 3.036689\ldots & 11.30930\ldots\\ 
$89$ & 352440 & 3.036100\ldots & 11.30056\ldots\\ 
$97$ & 456288 & 3.031998\ldots & 11.26935\ldots\\ 
$101$ & 515100 & 3.032463\ldots & 11.26755\ldots\\ 
$103$ & 546312 & 3.030308\ldots & 11.25228\ldots\\ 
$107$ & 612468 & 3.028370\ldots & 11.23855\ldots\\ 
$109$ & 647460 & 3.029877\ldots & 11.24644\ldots\\ 
$113$ & 721392 & 3.028016\ldots & 11.23330\ldots\\ 
$127$ & 1024128 & 3.024393\ldots & 11.20309\ldots\\ 
$131$ & 1123980 & 3.024148\ldots & 11.19945\ldots\\ 
$137$ & 1285608 & 3.022889\ldots & 11.19063\ldots\\ 
$139$ & 1342740 & 3.022686\ldots & 11.18747\ldots\\ 
$149$ & 1653900 & 3.020586\ldots & 11.17269\ldots\\ 
\end{tabular}
\end{table}
\par

\begin{table}[p]
\centering
\caption{Chebotarev invariants of some other groups}
\begin{tabular}{c|c|c|c}
Name & Order & $\cheb{G}$ & $\scheb{G}$ \\
\hline
  $W(G_2)=D_{12}$ & 12 & 4.315\underline{15}\ldots= 717/165&23.45407\ldots\\
  $W(C_4)$& 384 &  4.864890\ldots & 29.10488\ldots\\
  $W(F_4)$ & 1152 & 5.417656\ldots&35.12470\ldots\\
  $GL(2,\Fp_7)$ & 2016 & 3.767768\ldots&17.29394\ldots\\
  $A_5\times A_5$ & 3600 & 5.374156\ldots&35.41628\ldots\\
  $W(C_5)$ & 3840 &  4.863533\ldots & 28.13517\ldots \\
  $M_{11}$ & 7920 & 4.850698\ldots & 29.72918\ldots\\
  $GL(3,\Fp_3)$ & 11232 & 4.110394\ldots&22.77077\ldots\\
  $G_2(\Fp_2)$ & 12096 & 5.246204\ldots & 34.24515\ldots\\
  $Sz(8)$ & 29120 & 3.101639\ldots & 11.92233\ldots\\
  $W(C_6)$ & 46080& 5.792117\ldots & 39.56093\ldots\\
  $W(E_6)$ &51840& 4.470824\ldots&23.93050\ldots\\
  $Sp(4,\Fp_3)$ & 51840 & 4.401859\ldots&24.03143\ldots\\
  $PGL(3,\Fp_4)$ & 60480 & 3.763384\ldots & 19.49865\ldots\\
  $M_{12}$ & 95040 & 4.953188\ldots & 29.53947\ldots\\
  $J_1$ & 175560 & 3.423739\ldots & 14.76364\ldots\\
  $M_{22}$ & 443520 & 4.164445\ldots & 22.70981\ldots\\
  $J_2$ & 604800 & 4.031298\ldots & 19.07590\ldots\\
  $W(C_7)$ & 645120 &  4.632612\ldots & 25.54504\ldots\\
  $PSp(6,\Fp_2)$ & 1451520 & 5.270439\ldots & 34.84139\ldots\\
  $W(E_7)$ & 2903040& 5.398250\ldots&36.04850\ldots\\
  $G_2(\Fp_3)$ & 4245696 & 4.511630\ldots& 24.06106\ldots\\
  $M_{23}$ & 10200960 & 4.030011\ldots & 20.98580\ldots\\
  $W(C_8)$ & 10321920 &4.928996\ldots &  28.53067\ldots\\
  $T$ &  17971200 & 4.963701\ldots & 32.54160\ldots\\
  $Sz(32)$ & 32537600 & 2.755449\ldots & 9.107751\ldots\\
  $HS$ & 44352000 & 4.484432\ldots& 25.68549\ldots\\ 
  $J_3$ & 50232960 & 4.304616\ldots & 23.42082\ldots\\
  $W(C_9)$ & 185794560 & 4.716359\ldots & 26.41344\ldots\\
  $M_{24}$ & 244823040 & 4.967107\ldots & 29.84845\ldots\\
  $Sp(4,\Fp_7)$ & 276595200 &  3.501127\ldots & 14.83811\ldots \\
$\Omega^+(4,\Fp_{31})$ &   442828800& 3.829841\ldots& 17.60003\ldots\\
$\Omega^-(4,\Fp_{31})$ &   443751360& 3.003133\ldots& 11.02613\ldots\\
  $W(E_8)$& 696729600 & 4.194248\ldots&20.79438\ldots\\
  $McL$ & 898128000 & 4.561453\ldots& 27.45649\ldots\\
  $Sp(4,\Fp_{9})$ &3443212800& 3.409108\ldots & 14.04475\ldots\\
  $He$ & 4030387200 & 3.488680\ldots & 14.31119\ldots\\
  $G_2(\Fp_5)$ & 5859000000 & 3.855868\ldots& 18.68766\ldots\\
  $Sp(6,\Fp_3)$& 9170703360 & 3.871692\ldots & 18.90072\ldots  \\
  $Co_3$ & 495766656000 & 4.535119\ldots & 25.99974\ldots \\
  $Co_2$ &  42305421312000 & 3.865290\ldots & 17.74829\ldots \\
  $\Omega(5,\Fp_{31})$ & 409387254681600 & 3.277801\ldots& 12.90986\ldots\\
  $Rub$ & 43252003274489856000 & 5.668645\ldots & 36.78701\ldots\\
\end{tabular}
\end{table}

\begin{figure}[p]
\centering
\caption{Empirical distribution of the waiting time for $H_{31}$}
\includegraphics[width=4in]{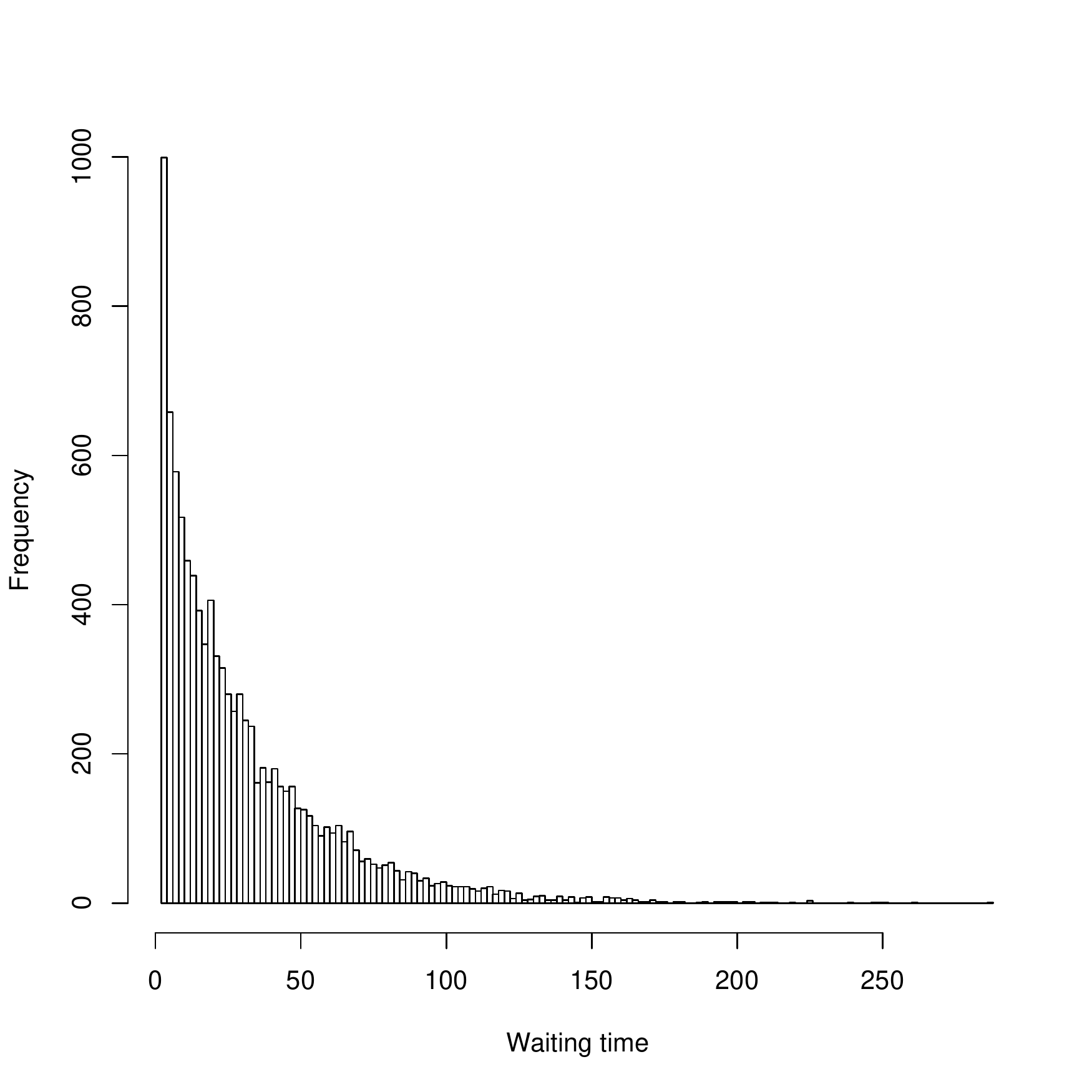}
\end{figure}


\begin{figure}[p]
\centering
\caption{Distribution of the Chebotarev invariant for groups of order $720$}
\includegraphics[width=4in]{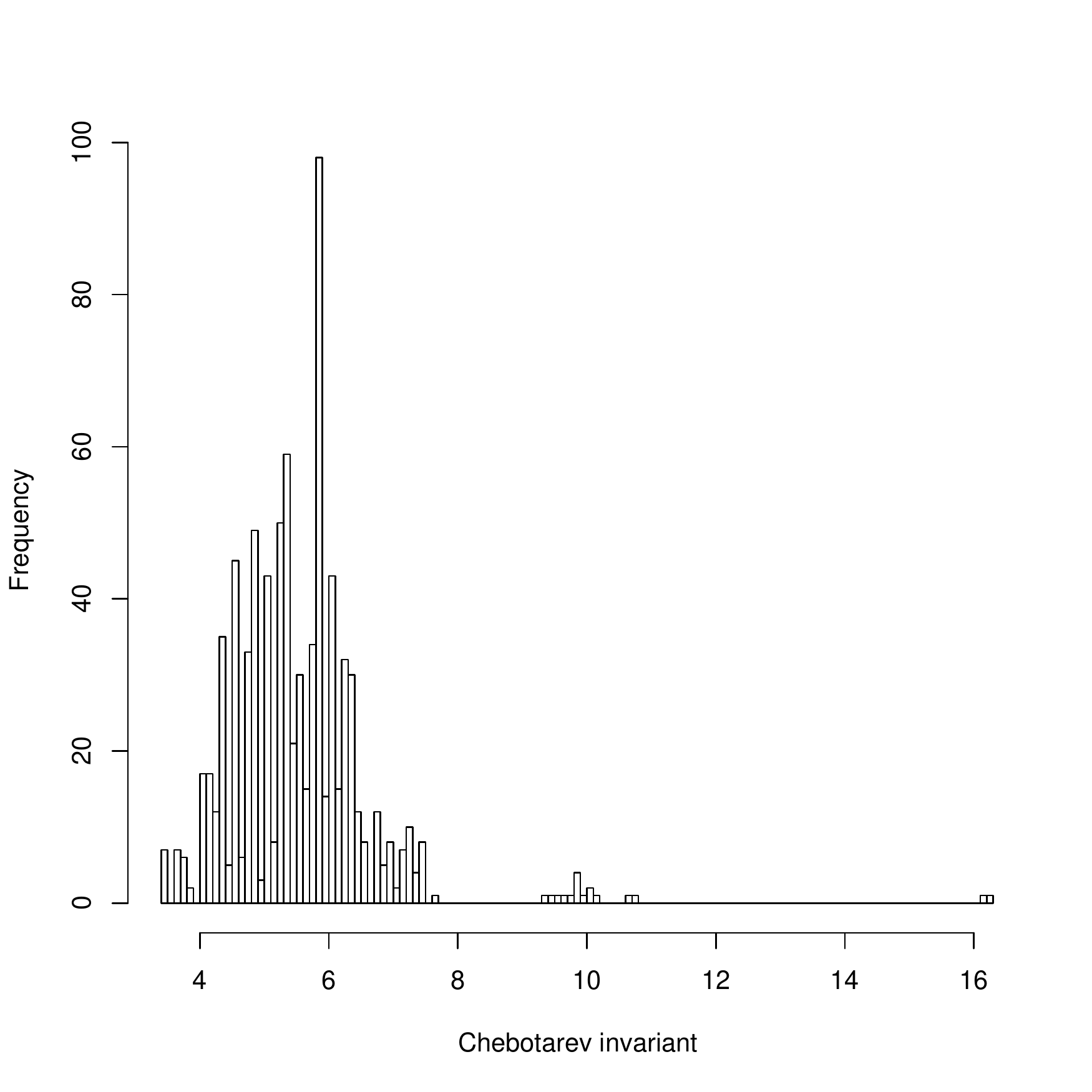}
\end{figure}




\section{Arithmetic considerations}\label{sec-arith}

In this short section, we indicate the (expected) number-theoretic
connections of our work.
\par
First, let $K$ be a Galois extension of $\Qq$ with group $G$.  For
each prime $p$ that is unramified in $K$, we have a well-defined
Frobenius conjugacy class $\frob_{p,K}\in G^{\sharp}$. For simplicity,
we denote $\frob_{p,K}=1$ when $p$ is ramified in $K$.  The
\emph{Chebotarev density theorem} says that
\begin{equation}\label{eq-cheb-density}
\lim_{y\to +\infty} \frac{|\{p \leq y : \frob_{p,K}=C\}|}{\pi(y)} =
\frac{|C|}{|G|}
\end{equation}
where $C\in G^{\sharp}$ is a fixed conjugacy class of $G$ and $\pi(y)$
is the usual prime-counting function, i.e., the number of primes
$p\leq y$.
\par
Now fix a real number $y$ large enough, so that every conjugacy class
of $G$ is of the form $\frob_{p,K}$ for some $p\leq y$.  For each
$i\geq 1$, select uniformly and independently a random prime $p$ from
the set $\{p : p\leq y\}$ and define $X_{i,y}^\sharp = \frob_{p,K}$.
We thus have a sequence of independent and identically distributed
random variables $X(y)=(X_{i,y}^{\sharp})$ in $G^{\sharp}$.  As usual,
we define the waiting time
\[
\tau_{X(y),G}=\min\{n\geq 1\,\mid\,
(X_{1,y}^{\sharp},\ldots,X_{n,y}^{\sharp})\text{ generate $G$}\}\in
[1,+\infty].
\]
\par
Using the Chebotarev density theorem, one obtains easily
$$
\lim_{y\to + \infty} \expect\big(\tau_{X(y),G}\big) = \cheb{G}.
$$
\par
Therefore, in an imprecise way, $c(G)$ can also be thought of as the
expected number of ``random'' primes $p$ needed for $\frob_{p,K}$ to
generate $G=\Gal(K/\Qq)$. Indeed, this is our motivation for using the
name ``Chebotarev invariant''.
\par
Of course in practice, one usually considers the (non-random) sequence
$\frob_{2,K}$, $\frob_{3,K}$, $\frob_{5,K}$, $\frob_{7,K},\ldots$. We
now explain, informally, what can be expected to happen in that
situation. The deterministic analogue of the Chebotarev waiting time
is given
$$
\tau(K)=\min\{k\geq 1\,\mid\,\text{the first $k$ conjugacy classes
  $\frob_{2,K}$, \ldots, $\frob_{p_k,K}$ generate $G$}\},
$$
where $p_k$ is the $k$-th prime number. 
\par
However, for a fixed $K/\Qq$, the value of $\tau(K)$ might diverge
considerably from $\cheb{G}$.  So we suppose we have some family
$\mathcal{K}$ of finite Galois extensions of $\Qq$ (or another base
field), all (or almost all) of which have Galois group
$\Gal(K/\Qq)\simeq G$, a fixed finite group, and that, for all values
of some parameter $x\geq 1$, we have finite subfamilies
$\mathcal{K}_x$ (which exhaust $\mathcal{K}$ as $x\ra +\infty$) and
some averaging process for invariants of the fields in $\mathcal{K}$,
denoted $\expect_x$ (for instance, one might take
$$
\expect_x(\alpha(K))=\frac{1}{|\mathcal{K}_n|}\sum_{K\in\mathcal{K}_x}{
\alpha(K)
}
$$
but other weights, involving multiplicities, etc, might be better
adapted). 
\par
Using this, we can define Chebotarev invariants for the family
$\mathcal{K}$ by averaging:
$$
\cheb{\mathcal{K}_x}=\expect_x(\tau(K)),\quad\quad
\scheb{\mathcal{K}_x}=\expect_x(\tau(K)^2).
$$
\par
The basic arithmetic question is then: for a given family, is it true
that $\cheb{\mathcal{K}_x}$ is, for $x$ sufficiently large at least,
close to $\cheb{G}$ (and similarly for the secondary Chebotarev
invariant)? The basic reason one can expect this is the Chebotarev
density theorem~(\ref{eq-cheb-density}).  
We want to point out a few difficulties that definitely arise in
trying to make this precise.
\par
First of all, quantifying the Chebotarev density theorem is
\emph{hard}: it almost immediately runs into issues related to the
Generalized Riemann Hypothesis; even in the seemingly trivial case
where $G=\Zz/2\Zz$ (quadratic extensions of $\Qq$), the basic question
of estimating the size of the smallest prime $p$ for which the
corresponding Frobenius is non-trivial, i.e., the smallest quadratic
non-residue modulo $p$, in terms of the discriminant of the field, is
unsolved (see, e.g.,~\cite[Prop. 5.22, Th. 7.16]{ant} for conditional
and uncontional results in that case). This is a problem because if we
sum with uniform weight, a single ``bad'' field $K_0$ can destroy any
chance of approaching the Chebotarev invariant. Indeed: note that in
that case
\begin{equation}\label{eq-bad}
\expect_x(\tau(K))\geq \frac{1}{|\mathcal{K}_x|}
k_{min}(K_0)
\end{equation}
where
$$
k_{min}(K)=\min\{k\geq 1\,\mid\, \frob_{p,K}\not=1\}
$$
is the index of the first non-trivial Frobenius conjugacy
class. In the current state of knowledge, it can be that there exists
$K_0$ with
$$
k_{min}(K_0)> \disc(K_0)^{A}
$$
for some constant $A>0$ (see~\cite{lmo}); on the other hand, if the
family $\mathcal{K}$ is defined as that of splitting fields of
monic polynomials of degree $n$, and the subfamily $\mathcal{K}_x$ is
that of polynomials of height $\leq x$, then we know that most
$K\in\mathcal{K}$ have Galois group $\mathfrak{S}_n$, that
$|\mathcal{K}_x|=(2x+1)^n$ if $x$ is an integer, and the discriminant
is obviously often also \emph{at least} a power of
$x$. Thus~(\ref{eq-bad}) might already be bad enough to preclude any
comparison. On the other hand, on the Riemann Hypothesis, we have
$$
k_{min}(K)\ll (\log \disc(K))^2,
$$
(where the implied constant depends on $G$), and the problem would
then be alleviated.
\par
Another issue is that one can not expect, as stated, to have
$$
\lim_{x\ra +\infty}{\cheb{\mathcal{K}_x}}=\cheb{G}
$$
for interesting families for the simple reason that the statistic of
small primes is typically \emph{not} the uniform one, i.e., if we fix
a prime $p$, we can not expect to have
$$
\lim_{x\ra
  +\infty}{\expect_x(\charfun_{\{\frob_{p,K}=c^{\sharp}\}})}=\nu_G(c^{\sharp}),
$$
even if we assume that all the fields involved are unramified at $p$.
\par
For instance, consider $\mathcal{K}$ the set of cubic polynomials
$$
X^3+a_2X^2+a_1X+a_0,
$$
with $a_i\in\Zz$, with $\mathcal{K}_x$ those where $|a_i|\leq x$ for
$i=0$, $1$, $2$, and count them uniformly. Take $p=5$ and consider
only polynomials with no repeated root modulo $5$ and splitting field
of degree $6$; then, asymptotically, the conjugacy Frobenius at $5$
will be distributed in $\mathfrak{S}_3=G$ as dictated by the
factorization of the polynomial modulo $5$. One finds easily that
there are $100$ monic cubic polynomials in $\Fp_5[X]$ with non-zero
discriminant (there are $25$ with repeated roots), among which:
\begin{itemize}
\item $10$  split in linear factors, i.e., a density $1/10$;
\item $40$ are irreducible, i.e., a density $4/10$;
\item $50$ split as a product of one linear factor and one irreducible
  quadratic factor, i.e., a density $1/2$.
\end{itemize}
\par
This is in sharp constrast with the density of the three corresponding
conjugacy classes in $\mathfrak{S}_3$, which are respectively:
\begin{itemize}
\item $1/6$ for the identity class;
\item $1/2$ for the $3$-cycles;
\item $1/3$ for the transposition.
\end{itemize}
\par
In particular, not even the relative frequencies are preserved! On the
other hand, it is well-known that if $p$ is increasing, the
discrepancy between the distribution of the factorization patterns of
squarefree polynomials modulo $p$ and the density of conjugacy classes
disappears: we have
$$
\frac{1}{p^n}|\{ f\in\Fp_p[X]\,\mid\, \text{$f$ squarefree of degree
  $n$ with $\frob_f=c^{\sharp}$} \}| \sim \nu_{G}(c^{\sharp})
$$
uniformly for all conjugacy classes $c^{\sharp}\in G=\mathfrak{S}_n$.
\par
This suggests that it is likely that one can prove some relevant
results: one would consider some increasing starting point $s(x)\geq
2$ and a modified waiting time
$$
\tau_x(K)=\min\{k\,\mid\, \text{the first $k$ conjugacy classes
  $\frob_{p,K}$ with $p\geq s(x)$ generate $G$}\}
$$
and hope to prove (possibly under the Generalized Riemann Hypothesis,
possibly unconditionally after throwing away a few ``bad'' fields)
that
$$
\lim_{x\ra +\infty}{\expect_x(\tau_x(K))}=\cheb{G},
$$
for suitable $s(x)$. One may guess that for polynomials of height
$\leq x$ and fixed degree $n$ (and $G=\mathfrak{S}_n$), this would be
true with $s(x)\asymp \log x$.
\par



\section{Remarks and problems}

We finish with a few more remarks and problems.
\par
\begin{itemize}
\item (What does the invariant ``know''?) As a bare numerical
  invariant of a finite group, the Chebotarev invariant seems to be
  fairly subtle. For instance, we see from Section~\ref{sec-ex1} that
  it ``knows'' that vector spaces over finite fields are in some sense
  very similar for varying base field, but that they become also
  ``simpler'' as the cardinality of the base field grows. It also
  seems to know that non-reductive finite matrix groups are
  worse-behaved than reductive ones (as shown by the results for
  $H_q$). What else does the invariant reveal?
\item (A method for upper bounds) There are, in the literature, quite
  a few results about a finite group $G$ of the type: ``if a subgroup
  $H$ contains elements in some set $C_1$, some other set $C_2$, ...,
  some other set $C_m$, of conjugacy classes, then $H$ is in fact
  equal to $G$''. For instance, a lemma of Baer quoted by
  Gallagher~\cite[Lemma, p. 98]{gallagher} says that there is no
  proper subgroup $H$ of $\sy_n$ which (1) contains an $n$-cycle, (2)
  contains a product of a transposition and cycles of odd lengths, (3)
  contains an element of order divisible by a \emph{prime}
  $p>n/2$. Another such result is the Lemma~\ref{lm-serre} of Serre
  for $SL(2,\Fp_p)$, and we also mention~\cite[Lemma 3.2]{jkz} for
  another example with the Weyl group $W(E_8)$, and there are many
  other such results used, e.g., for proving concrete cases of
  Hilbert's Irreducibility Theorem.
\par
With this notation, and assuming we work with a sequence of
independent and uniformly distributed $G$-valued random variables
$(X_n)$, this means that we have
$$
\tau_{G}\leq \tau_{C_1,\ldots,C_m}=\max(\tau_{C_j},\ 1\leq j\leq m),
$$
where
$$
\tau_{C_j} =\min\{n\geq 1\,\mid\, X_n^{\sharp}\in C_j\}.
$$
\par
From Proposition~\ref{pr-collect}, we obtain easily an upper bound
\begin{equation}\label{eq-dist}
\cheb{G}\leq \expect(\tau_{C_1,\ldots,C_m})=
\sum_{\emptyset\not=I\subset \{1,\ldots, m\}}{\frac{(-1)^{|I|+1}}{
    \nu\Bigl(\bigcup_{j\in I}{C_j}\Bigr) }},
\end{equation}
and one may hope to approximate $\cheb{G}$ by choosing wisely the sets
$(C_j)$. 
\par
However, it is not clear at all to what extent this can approach the
truth. Here are some examples:
\par
(1) Baer's lemma gives only an upper bound
$$
\cheb{\sy_n}\ll n
$$
as $n\ra +\infty$, which is quite weak compared with
Theorem~\ref{th-cheb-symmetric} (it is dominated by the time required
to obtain an $n$-cycle). How far is this from the best possible result
that can be obtained in this way, and how far is the latter from
Theorem~\ref{th-cheb-symmetric}?
\par
(2) Consider $G=\Fp_p^2$ with $p$ odd. It is possible to take
\begin{align*}
C_1&=\{(x,y)\in \Fp_p^2-\{0\}\,\mid\, y\not=0\text{ and $xy^{-1}$ is a
  square in $\Fp_p$}\},
\\
C_2&=\Fp_p^2-\{0\}-C_1.
\end{align*}
\par
The point is that whenever $(v,w)\in C_1\times C_2$, $w$ and $v$ are
not on the same line through the origin, so $(v,w)$ generate
$\Fp_p^2$. Since 
$$
|C_1|=|C_2|=(p^2-1)/2,\ |C_1\cup C_2|=p^2-1,
$$ 
this leads to
$$
\cheb{\Fp_p^2}\leq \frac{1}{\nu(C_1)}+\frac{1}{\nu(C_2)}-
\frac{1}{\nu(C_1\cup C_2)}=\frac{3p^2}{p^2-1},
$$
which asymptotically requires one more step on average than the right
Chebotarev invariant (given by~(\ref{eq-cheb-vs})), namely
$\cheb{\Fp_p^2}=(2p^2+p)/(p^2-1)$. It seems also that this type of
sets is essentially best possible for applying this upper bound in
this case. 
\par
(3) Consider $G=W(E_8)$, the Weyl group of $E_8$. There is a
non-trivial homomorphism
$$
\eps\,:\, W(E_8)\ra \{\pm 1\},
$$
and in~\cite[Lemma 3.2]{jkz}, jointly with F. Jouve, we proved that
one could take $C_1=\ker(\eps)$, $C_2$ the union of the conjugacy
classes of $w$ and $w^2$, where $w\in W(E_8)$ is a Coxeter element;
the density of $C_2$ is $1/15$ and we then get the upper-bound
$$
2+15-30/17=25.23\ldots
$$
instead of the correct value $4.194248\ldots$.
\par
(4) For $SL(2,\Fp_p)$, Theorem~\ref{th-psl} shows that the sets $C_1$,
$C_2$, $C_3$ given by Lemma~\ref{lm-serre} give an asymptotically
optimal answer (and this is an essential ingredient in the proof).
\par
Despite this relative inefficiency, it is interesting to notice that
in applications of sieve methods to probabilistic Galois theory (as
was the case in~\cite{gallagher}) and~\cite{jkz},\footnote{\ If only
  implicitly in the latter.} it is this type of distinguishing
families which can be used in estimating how rare ``small'' Galois
groups are in certain families, and in fact it is the quantity
\begin{equation}\label{eq-naive}
\sum_{i=1}^m{\frac{1}{\nu(C_i)}}
\end{equation}
which occurs naturally as coefficient in a ``saving factor'' of the
large sieve; see, e.g,~\cite[p. 57]{large-sieve}, where the question
of minimizing this by varying the sets was raised explicitly for
symmetric groups.
\item (General upper bounds?) A first problem is to bound
$\cheb{G}$ from above, in a meaningful
way. Since we have
$$
\tau_{G}\leq \sum_{\Hc\in\max(G)}{\hat{\tau}_{\Hc}}.
$$
we obtain
$$
\cheb{G}\leq \sum_{\Hc\in\max(G)}{\frac{1}{1-\nu(\Hc^{\sharp})}},
$$
from~(\ref{eq-comput}). Together with~(\ref{eq-maximal}), this gives
an upper bound
\begin{equation}\label{eq-trivial}
\cheb{G}\leq |G|\sum_{\Hc\in\max(G)}{\frac{1}{|\Hc|}}
\end{equation}
which is close to being sharp for the groups $H_q$ of
Section~\ref{sec-solvable}: indeed, if $q$ is odd, then
Lemma~\ref{lm-max-hq} gives
\begin{align*}
|H_q|\sum_{\Hc\in\max(H_q)}{\frac{1}{|\Hc|}}&=q(q-1)\Bigl(\frac{1}{q-1}
+\sum_{\ell\mid q-1}{\frac{\ell}{q(q-1)}}\Bigr)\\
&=
q+\sum_{\ell\mid q-1}{\ell}=q+2+\sum_{2<\ell\mid q-1}{\ell}
\end{align*}
(where $\ell$ runs over prime divisors of $q-1$). If $q=2\ell+1$
($\ell$ odd prime) is a Sophie Germain prime, this gives
$$
q+2+\sum_{2<\ell\mid q-1}{\ell}=q+2+\frac{q-1}{2}=\frac{3(q+1)}{2},
$$
which is off, asymptotically, only by a factor $3/2$ from the value
$$
\cheb{H_q}=q+O(q^{-1})
$$
that follows from~(\ref{eq-caxpb}). Of course, it is not known that
there are infinitely many Sophie Germain primes, but for
$q=2\ell_1\ell_2+1$, with $\ell_i$ prime, we have
\begin{align*}
|H_q|\sum_{\Hc\in\max(H_q)}{\frac{1}{|\Hc|}}&=
\begin{cases}
q+\ell_1+\ell_2+2&\text{ if } \ell_1\not=\ell_2\\
q+\ell_1+2&\text{ if } \ell_1=\ell_2
\end{cases}
\\
&\leq 2q.
\end{align*}
\par
By sieve methods, it is known that there are infinitely many primes
$q$ for which either $q$ is a Sophie Germain prime, or is
$2\ell_1\ell_2+1$, and hence one sees that the ``trivial''
estimate~(\ref{eq-trivial}) above can not be improved by more than a
constant in full generality.  On the other hand, it is very far off
for many groups: for a random example, it gives
$$
4.7820\ldots=\cheb{A_7}\leq 93.
$$
\par
It would be more interesting to have a decent upper bound in terms of
the order of $G$ only. Here, using the set of all conjugacy classes
in~(\ref{eq-dist}), we get as an upper bound from the contribution of
singletons that
$$
\cheb{G}\leq \sum_{g^{\sharp}\in
  G^{\sharp}}{\frac{1}{\nu(g^{\sharp})}} 
=\sum_{g^{\sharp}\in
  G^{\sharp}}{|N_G(g)|},
$$
(where $N_G(g)$ is the normalizer of $g$ in $G$). This gives trivially
$$
\cheb{G}\leq |G|^2,
$$
but this seems unlikely to be close to the truth (for $G\not=1$).  For
instance, since
$$
\cheb{H_q}=q\sim \sqrt{|H_q|},
$$
one may wonder if $H_q$ is also (essentially) extremal in
this sense, i.e., one may ask whether an estimate
$$
\cheb{G}\ll \sqrt{|G|}
$$
holds for all $G$. (Certainly for $|G|=q(q-1)$ with $q\leq 43$ prime,
it is experimentally true that $H_q$ maximizes the Chebotarev
invariant). 
\item (Other classes of groups?) There are many classes of groups for
  which it should be possible to understand the behavior of the
  Chebotarev invariant, at least asymptotically. For instance, one can
  consider non-reductive subgroups of finite matrix groups, e.g., the
  standard Borel subgroup (upper triangular matrices) of
  $GL(n,\Fp_q)$. In fact, solvable groups seem particularly
  interesting.
\end{itemize}



\section*{Appendix: Magma script}


The following script can be used to compute the Chebotarev invariant
(and the secondary invariant) using \textsc{Magma}, by applying the
formulas~(\ref{eq-comput}) and~(\ref{eq-secondary}). The output is
given as real approximations since usually the denominators are
unwieldy. Also note that because of the use of the construct
\texttt{Subsets(M)}, this script only applies to groups with at most
$29$ conjugacy classes of maximal subgroups;\footnote{\ For
  alternating groups $A_n$, this means $n\leq 47$, or $n\in
  \{49,51,53\}$.} to -- try to -- compute further, one would have to
replace the loop over subsets obtained in this manner with a
hand-rolled one.
\par
A similar \textsc{GAP} script is available upon request, as well as a
\textsc{Sage} version, which basically calls the \textsc{GAP} group
theory routines. However, these versions are much slower.
\par
The last routine in the script is useful for ``empirical'' study of
the probabilistic model.
\par
\lstset{basicstyle=\footnotesize\ttfamily}
\begin{lstlisting}[columns=flexible]
// The following calculates J such that 
// J[k][i]=true if and only if the k-th maximal subgroup 
// of G intersects the i-th conjugacy class of G

MCIntersectionMatrix:=function(G,C,f,M)
  J := [ [false : i in [1..#C]] : k in [1..#M] ];
  for k in [1..#M] do
    H := M[k]`subgroup;
    CH := ConjugacyClasses(H);
    for j in [1..#CH] do
      J[k][f(CH[j][3])] := true;
    end for;
  end for;
  return J;
end function;

// This returns [c,s] where c is the Chebotarev invariant of G
// and s the secondary invariant.

Chebotarev:= function (G)
  if IsTrivial(G) then
    return <1.0,1.0>;
  end if;

  C := ConjugacyClasses(G);
  f := ClassMap(G);
  M := MaximalSubgroups(G);
  J := MCIntersectionMatrix(G,C,f,M);
  c:=0.0; s:=0.0;

  for I in Subsets({1..#M}) do
    if #I ne 0 then 
      v:=0;
      for i in [1..#C] do
        if forall(t) {k: k in I | J[k][i]} then
	  v:= v + C[i][2]/#G;
	end if;
      end for;
      c := c + (-1)^(#I+1)/(1-v);
      s := s+ (-1)^(#I)/(1-v)*(1-2/(1-v));
    end if;
  end for;
  return([c,s]);

end function;

// Compute empirical Chebotarev invariant.
// The optional parameter steps is the number
// of iterations to do. Example:
// > EmpiricalChebotarev(Alt(7):steps:=10000);

EmpiricalChebotarev:=function(G : steps:=1)
  total:=0;
  C := ConjugacyClasses(G);
  f:=ClassMap(G);
  M := MaximalSubgroups(G);
  J := MCIntersectionMatrix(G,C,f,M);
  for count in [1..steps] do
    nb:=0;
    vprint User1: "Iteration, ", count;
    // Start with all subgroups
    possible:=[ 1..#M ];
    while possible ne []  do 
      g:=Random(G);
      nb := nb+1;
      index:=f(g);
      // Only those subgroups containing the class of g remain
      possible:=[ k : k in possible | J[k][index]  ];
    end while;
    total:=total+nb;
  end for;
  return total/steps, total/steps*1.0;
end function;
\end{lstlisting}

\end{document}